\documentclass[11pt, a4 paper]{amsart} 
\pdfoutput=1	

\usepackage[psamsfonts]{amssymb} 
\usepackage{amsfonts,amsmath,mathabx}
\usepackage{color}

\usepackage{tikz}
\usetikzlibrary{positioning}
\usetikzlibrary{calc} 
\usepackage{subfigure}
\usepackage[T1]{fontenc}
\usepackage[utf8x]{inputenc}
\usepackage{amsthm}
\usepackage{thmtools, thm-restate}
\usepackage{centernot}
\usepackage{amsfonts}
\usepackage[polish,english]{babel}
\usepackage{graphicx}
\usepackage{wrapfig}
\usepackage{mathtools}
\usepackage{stmaryrd}
\usepackage{enumitem}
\setlistdepth{9}
\usepackage{tipa}
\usepackage[margin = 2cm]{caption}
\usepackage{hyperref}
\usepackage[capitalise]{cleveref}
\usepackage{accents}
\usepackage[thinlines]{easytable}
\usepackage{pgf,tikz,pgfplots}
\usepackage{mathrsfs}
\usetikzlibrary{arrows,calc}
\usepackage{tikz-cd}
\usepackage{tabularx}
\usepackage{import}
\usepackage[thinlines]{easytable}

\crefname{thmA}{theorem}{theorems}
\Crefname{thmA}{Theorem}{Theorems}

\crefname{corA}{corollary}{corollaries}
\Crefname{corA}{Corollary}{Corollaries}

\numberwithin{equation}{section} 

\newtheorem{theorem}{Theorem}[section]

\newtheorem{proposition}[theorem]{Proposition}
\newtheorem{lemma}[theorem]{Lemma}
\newtheorem{corollary}[theorem]{Corollary}

\theoremstyle{remark}
\newtheorem{remark}[theorem]{Remark}
\newtheorem{remarks}[theorem]{Remarks}
\newtheorem{example}[theorem]{Example}

\theoremstyle{definition}
\newtheorem{definition}[theorem]{Definition}

\def\Z{\mathbb Z}
\def\N{\mathbb N}
\def\R{\mathbb R}
\def\Q{\mathbb Q}

\def\G{\Gamma}

\def\e{\varepsilon}

\def\-{\overline}
\def\wh{\widehat}

\def\aut{{\rm{Aut}}}
\def\autN{{\rm{Aut}}(F_N)}
\def\out{{\rm{Out}}} 
\def\outN{{\rm{Out}}(F_N)} 
\def\tsc{{\rm{\bf{fsc}}}}
\def\GG{\mathbb{G}}
\def\mc{\mathcal}

\def\sym{{\rm{sym}}}

\def\st{\colon} 
\def\d{\delta}
\def\t{\tau}
\def\-{\overline}
\def\b{\beta}
\def\La{\Lambda}

\def\ad{{\rm{ad}}}

\def\H{\mathbb{H}}

\def\G{\Gamma}

\def\<{\langle}
\def\>{\rangle}

\begin{document}

\title[Profinite rigidity for free-by-cyclic groups with centre]
{Profinite rigidity for free-by-cyclic groups with centre}

\author[Martin R. Bridson and Pawe\l{} Piwek]
{Martin R.~Bridson and Pawe\l{} Piwek}
\address{Mathematical Institute\\
	Andrew Wiles Building\\
	ROQ, Woodstock Road\\
	Oxford\\
	United Kingdom}
\email{bridson@maths.ox.ac.uk,\ piwek@maths.ox.ac.uk}

\date{Minor edits 30 Sept 2024}


\begin{abstract} 
	A free-by-cyclic group $F_N\rtimes_\phi\Z$ has non-trivial centre if and only if 
	$[\phi]$ has finite order in ${\rm{Out}}(F_N)$. 
	We establish a profinite ridigity result for such groups: 
	if $\Gamma_1$ is a free-by-cyclic group with non-trivial centre 
	and $\Gamma_2$ is a finitely generated free-by-cyclic group 
	with the same finite quotients as $\Gamma_1$, 
	then $\Gamma_2$ is isomorphic to $\Gamma_1$. 
	One-relator groups with centre are similarly rigid. 
	We prove that finitely generated free-by-(finite cyclic) groups 
	are profinitely rigid in the same sense; 
	the proof revolves around a finite poset $\tsc(G)$ 
	that carries information about the centralisers of finite subgroups of $G$ 
	-- it is a complete invariant for these groups.  
	These results provide contrasts with 
	the lack of profinite rigidity among 
	surface-by-cyclic groups
	and (free abelian)-by-cyclic groups, 
	as well as general virtually-free groups. 
\end{abstract}

\maketitle

\section{Introduction}
\label{s:introduction}

It is unknown whether free-by-cyclic groups $F\rtimes\Z$, 
with $F$ finitely generated, 
can be distinguished from each other by their finite quotients. 
Likewise, it is unknown whether 
residually-finite 1-relator groups 
can be distinguished from each other by their finite quotients. 
In this article we will resolve these questions 
in the case of groups with non-trivial centre. 
For free-by-cyclic groups, the precise statement is as follows:

\begin{restatable}{thmA}{thmmain}
\label{t:main}
	If a pair of free-by-cyclic groups 
	$\Gamma_1=F_{N}\rtimes_{\phi_1}\Z$ and $\Gamma_2=F_{M}\rtimes_{\phi_2}\Z$ 
	have the same finite quotients 
	and $[\phi_1]\in{\rm{Out}}(F_N)$ has finite order, 
	then $[\phi_2]\in{\rm{Out}}(F_M)$ has finite order 
	and $\Gamma_1\cong\Gamma_2$.
\end{restatable}

In interpreting this theorem, it is important to note that 
$\G_1\cong\G_2$ does not imply that $N=M$.
Moreover, when we restrict to the case $N=M$, 
the existence of an isomorphism 
$F_{N}\rtimes_{\phi_1}\Z\cong F_{N}\rtimes_{\phi_2}\Z$ 
does not imply that 
$\Phi_1 = [\phi_1]$ is conjugate to $\Phi_2 = [\phi_2]$ 
in ${\rm{Out}}(F_N)$, 
nor does it imply that 
$\<\Phi_1\>$ is conjugate to $\<\Phi_2\>$. 
If $N=M$ then we can deduce that 
$\Phi_1$ and $\Phi_2$ have the same order in ${\rm{Out}}(F_N)$ 
but if $N\neq M$ then this too may fail. 
These phenomena are discussed in \cref{s:isomorphism}.

In \Cref{t:main}, 
the quotient of each group by its centre 
will be free-by-(finite cyclic), 
but it will not be a split extension in general. 
The fact that these quotients are virtually free 
is not sufficient to guarantee that 
they are uniquely determined by their finite quotients, 
but we shall prove that 
free-by-(finite cyclic) groups are more rigid. 
This is a key step in our proof of \Cref{t:main} 
and it is of independent interest. 

\begin{restatable}{thmA}{thmfinitecyclic}
\label{thm:free-by-finite-cyclic}
	If a pair of finitely generated free-by-(finite cyclic) groups $G_1$ and $G_2$
	have the same finite quotients, then $G_1\cong G_2$. 
\end{restatable}

A crucial idea in the proof of \Cref{thm:free-by-finite-cyclic} 
is that free-by-(finite cyclic) groups $G$ 
are determined up to isomorphism by a natural poset $\tsc(G)$ 
built from the finite subgroups of $G$; 
see \Cref{thm:same_posets_implies_iso}. 

Before discussing our approach to these results, 
we review some of the context in which they sit. 

A group $\G$ from a class $\mathcal{C}$ is said to be 
{\em profinitely rigid among groups in $\mathcal{C}$} 
if any group $\La\in\mathcal{C}$ 
with the same finite quotients as $\G$ 
must be isomorphic to $\G$. 
If the groups in $\mathcal{C}$ are finitely generated, 
this can be rephrased as 
$\wh{\La}\cong\wh{\G} \Longrightarrow \La\cong\G$, 
where $\wh{\G}$ denotes the profinite completion of $\G$, 
i.e. the limit of the inverse system 
formed by all of its finite quotients. 
In recent years there have been a lot of advances 
in the understanding of profinite rigidity 
among groups that arise 
in low-dimensional geometry and topology. 
This provides motivation and context 
for our study of free-by-cyclic groups, 
in light of the fruitful and much-studied 
analogies between automorphisms of free groups
and automorphisms of surface groups 
and free-abelian groups -- 
see \cite{bestvina2003topology, bridson2006automorphism}.

There are several noteworthy works on profinite rigidity for
surface bundles and $n$-torus bundles over the circle.
For $2$-torus bundles over the circle, it is known that 
profinite rigidity can fail: 
Funar \cite{funar2013torus} pointed out that 
classical results of Stebe \cite{stebe1972conjugacy} 
provide infinitely many pairs of hyperbolic matrices 
$\phi_1,\phi_2\in{\rm{SL}}(2,\Z)$ 
such that the corresponding torus bundles over the circle 
have non-isomorphic fundamental groups 
$\G_1=\Z^2\rtimes_{\phi_1}\Z$ and $\G_2=\Z^2\rtimes_{\phi_2}\Z$
with $\wh{\G}_1\cong \wh{\G}_2$. 
In contrast, 
Bridson, Reid and Wilton \cite{bridson2017profinite} proved that 
the fundamental group of 
every once-punctured torus bundle over the circle 
is profinitely rigid in the class of 3-manifold groups. 
And Liu proved that at most finitely many 
hyperbolic 3-manifold groups 
can share the same profinite completion \cite{liu2023finite}, 
adding weight to the expectation 
that such groups are profinitely rigid, cf. 
\cite{bridson2021profinite}. 

The lack of profinite rigidity 
among groups of the form $\Z^N\rtimes_\phi\Z$ 
was first revealed by Serre 
in a seminal paper of 1964 \cite{serre1964exemples}. 
Baumslag \cite{baumslag1974residually} stripped the phenomenon 
identified by Serre down to its bare bones 
and identified pairs of distinct 
finite-by-cyclic groups that have the same finite quotients. 
Forty years later, 
Hempel \cite{hempel2014some} adapted a key idea from these works 
to construct pairs of 3-manifolds $M_1, M_2$ 
with $\H^2\times\R$ geometry such that 
the groups $\pi_1M_i = \Sigma\rtimes_{\phi_i}\Z$ are distinct 
but have the same profinite completion; 
these manifolds are closed-surface bundles over the circle, 
each with finite holonomy and the same fibre. 
Hempel also gave examples where 
the surface-fibres have non-empty boundary, 
but in this case he did not conclude that 
the fundamental groups (which now have the form $F\rtimes\Z$) 
were not isomorphic, but rather that 
there was no isomorphism between them 
that preserved the peripheral structure of the 3-manifold; 
it was unclear at the time 
whether this added restriction was necessary, 
but \Cref{t:main} shows that it is. 
 
All groups of the form $F_2\rtimes\Z$ arise as 
fundamental groups of punctured-torus bundles over the circle 
and so are covered by \cite{bridson2017profinite}, 
but when the rank $N$ of the free group is at least $3$, 
the groups $F_N\rtimes\Z$ form a more complicated class 
with more subtle geometry. 
In particular, it is unclear what to expect from them 
in terms of profinite rigidity. 
In the case where $F_N\rtimes\Z$ is hyperbolic, 
significant progress towards the analogue of Liu's result 
was made by Hughes and Kudlinska \cite{hughes2023profinite}, 
but profinite rigidity remains open 
even in the hyperbolic setting. 
In the light of Hempel's results, one might expect 
profinite rigidity to fail for groups $F_N\rtimes_\phi\Z$ 
where $\Phi = [\phi] \in{\rm{Out}}(F_N)$ has finite order, 
but \Cref{t:main} shows that this is not the case.

\medskip

Another rich class of groups of geometric dimension 2 
is formed by the {\em torsion-free 1-relator groups}. 
It remains unknown whether 
residually-finite 1-relator groups 
can be distinguished from one another 
by their profinite completions, 
but \Cref{t:1-relator} will allow us to settle this question 
in the case where the group has non-trivial centre. 

\begin{restatable}{thmA}{thmonerelator}
\label{t:1-relator} 
	Let $\G_1$ and $\G_2$ be residually finite groups 
	that admit 1-relator presentations. 
	If $\G_1$ has a non-trivial centre 
	and $\wh{\G}_1\cong\wh{\G}_2$, then $\G_1\cong\G_2$.
\end{restatable}

We shall see in \cref{s:1-relator} 
that this theorem would fail 
if we did not assume that $\G_2$ was residually finite. 
The passage from \Cref{t:main} to \Cref{t:1-relator} 
is made straightforward 
by a result of Baumslag and Taylor \cite{baumslag1967centre}, 
who proved that all 1-relator groups with centre 
are free-by-cyclic $F_N\rtimes\Z$. 
A further much-studied class 
in which the members with centre are free-by-cyclic 
is formed by the Generalised Baumslag-Solitar (GBS) groups, 
i.e. the fundamental groups of 
finite graphs of infinite cyclic groups \cite{levitt2015generalized}, an the analogue of  \cref{t:1-relator} holds in this
setting (Remark \ref{r:GBS}). 

GBS groups play a more important role in this article 
than the tone of this last remark might suggest: 
our proof of \Cref{t:main} begins with a change of perspective 
that replaces the visible semidirect decomposition 
$\G=F\rtimes\Z$ 
with a decomposition of $\G$ 
as a graph of infinite cyclic groups. 
We will argue that the groups $\G_1$ and $\G_2$ in \Cref{t:main} 
have the same profinite completion if and only if 
the groups modulo their centres
have the same profinite completion. 
The groups modulo their centres are free-by-(finite cyclic), 
so \Cref{thm:free-by-finite-cyclic} will tell us that 
$\G_1/Z(\G_1)$ is isomorphic to $\G_2/Z(\G_2)$ 
if these groups have the same finite quotients. 
But we need more from \Cref{thm:free-by-finite-cyclic} 
than just its statement: 
our proof of \Cref{thm:free-by-finite-cyclic} involves 
an analysis of graph-of-finite-groups decompositions 
of free-by-(finite cyclic) groups, 
and what we prove is that 
if the groups in \Cref{thm:free-by-finite-cyclic} 
have the same profinite completion, 
then one can transform the graph of groups 
representing the first into that representing the second 
by a sequence of {\em slide moves}. 
In the case $G_i=\G_i/Z(\G_i)$, 
these slide moves can be lifted 
to the graph-of-groups decompositions 
associated to the description of the $\G_i$ as GBS groups, 
proving that $\G_1\cong\G_2$.

There are several features of this outline of proof 
upon which we should expand. 
First, although \Cref{t:main} was our original motivation, 
\Cref{thm:free-by-finite-cyclic} emerges as a key result. 
In the case where 
the finite cyclic group is of prime-power order, 
\Cref{thm:free-by-finite-cyclic} 
was proved by Grunewald and Zalesskii -- 
see Theorem 3.12 of \cite{grunewald2011genus};
their proof relies on 
the subgroup structure of cyclic $p$-groups 
and uses a decomposition of $F_n$-by-$(\Z/p^\alpha)$ groups 
that is not available in the general case. 
It is important to note here that although one expects 
a proof of \Cref{thm:free-by-finite-cyclic} 
to use arguments about graphs of finite groups, 
it must avoid the fact that 
virtually free groups are not profinitely rigid in general, 
and hence one needs to use 
graphs with special features or enhancements. 

In our outline of proof, 
we spoke about the graph-of-groups decompositions 
of the groups $G_i$ 
in a manner that might suggest they are reasonably canonical. 
In fact, the uniqueness issue is somewhat delicate, 
as readers familiar with 
deformation spaces of group actions on trees will anticipate. 
Correspondingly, the main object of study 
in our proof of \Cref{thm:free-by-finite-cyclic} is 
the entirely canonical \emph{finite subgroups conjugacy poset}, 
denoted by $\tsc(G)$. 
The elements of this poset 
are the conjugacy classes of finite subgroups in $G$
and these classes are labelled 
by the first Betti number of the subgroup's centraliser 
as well as the size of the subgroup. 
As one might expect with such a natural object, 
versions of this poset have been studied elsewhere, 
for example in \cite{corson2023higman} 
where it is used to prove that graph products of finite groups 
are profinitely distinguished from each other 
(although in that setting it was not necessary 
to retain information about centralisers). 

\begin{restatable}{thmA}{thmposetinvariant}
\label{thm:same_posets_implies_iso}
	Let $G_1$ and $G_2$ be finitely generated 
	free-by-(finite cyclic) groups. 
	Then $G_1\cong G_2$ (equivalently, $\wh{G}_1\cong \wh{G}_2$) 
	if and only if there is a data-preserving isomorphism 
	of posets $\tsc(G_1)\cong\tsc(G_2)$. 
	
	Furthermore, the isomorphism $\tsc(G_1)\cong\tsc(G_2)$ 
	can be realised by a sequence of slide moves 
	transforming any realisation of $\tsc(G_1)$ 
	as a reduced graph of finite cyclic groups 
	into a realisation of $\tsc(G_2)$. 
\end{restatable}

The fact that the isomorphism $\tsc(G_1)\cong\tsc(G_2)$ 
can be exhibited through \emph{slide moves} 
is crucial for the passage 
from \Cref{thm:free-by-finite-cyclic} to \Cref{t:main}. 

\bigskip

The remainder of this paper is organised as follows. 
Following some preliminaries, 
in \Cref{sec:GBS} we study 
the GBS decomposition of free-by-cyclic groups with centre, 
explaining how the centre is encoded in this decomposition. 
In \Cref{sec:reduction} we use this structure 
to reduce \Cref{t:main} 
to the study of 
the profinite completions of free-by-(finite cyclic) groups, 
whose decompositions we study in \Cref{s:free-by-fc}. 
In \Cref{s:poset-slides} we introduce a key technical device---%
the poset associated to a numbered graph---%
and prove a result that will enable us 
to promote isomorphisms
to sequences of sliding moves in various contexts. 
In \Cref{sec:torsion} we translate these results into
more algebraic terms, introducing $\tsc(G)$ 
and proving \Cref{thm:same_posets_implies_iso}. 
Following an interlude in \Cref{sec:completions_detect} 
gathering some conclusions from  profinite Bass-Serre theory, 
we are able to complete the proof of our main theorems 
in \Cref{sec:proof-of-main-thms}. 
\Cref{t:1-relator} is proved in \Cref{s:1-relator} 
along with a companion result 
that illustrates the need to restrict attention 
to residually finite groups. 
In the final section we discuss 
some non-obvious isomorphisms $F_N\rtimes\Z\cong F_M\rtimes\Z$. 
\smallskip

\noindent{\bf Acknowledgements.} 
The second author was supported by 
the Mathematical Institute Scholarship 
from the University of Oxford.

\section{Preliminaries}

\subsection{Profinite completions} 

The reader will require only 
a rudimentary knowledge of profinite groups; 
when we require non-elementary facts we will provide references. 
The {\em profinite completion} $\wh{\G}$ 
of an abstract group $\G$ 
is the limit of the inverse system of finite quotients of $\G$; 
it is a compact topological group. 
The natural map $\G\to\wh{\G}$ is injective 
if and only if $\G$ is residually finite, 
and we identify $\G$ with its image. 
Every homomorphism of abstract groups $\phi:\G_1\to\G_2$ 
extends uniquely to a continuous homomorphism 
$\wh{\phi}:\wh{\G}_1\to\wh{\G}_2$. 

If $\G$ is finitely generated and $Q$ is finite, 
then every epimorphism $\wh{\G}\to Q$ is continuous 
by \cite{nikolov2003finite} 
and restriction to the dense subgroup $\G$ induces a bijection 
${\rm{Epi}}(\wh{\G},Q)\to {\rm{Epi}}({\G},Q)$. 
For each natural number $d$, 
there is a bijection between the subgroups of index $d$ in $\G$ 
and those of index $d$ in $\wh{\G}$; 
this assigns to each $\Lambda<\G$ 
its closure $\overline{\Lambda}<\wh{\G}$, 
and $\overline{\Lambda}\cong\wh{\Lambda}$. 
It follows that 
if $\G_1$ and $\G_2$ are finitely generated groups 
with $\wh{\G}_1=\wh{\G}_2$, 
then there is a 1-1 correspondence 
$\Lambda\leftrightarrow \overline{\Lambda}\cap \G_2$ 
between the subgroups of finite index in $\G_1$ 
and those in $\G_2$, 
such that corresponding subgroups have the same index 
and the same profinite completion.

\subsection{Graphs of groups}

We shall assume that the reader is familiar 
with the ideas and terminology of {\em Bass-Serre theory} 
\cite{serre1980trees}, 
and with the graph-of-spaces approach to the subject 
developed by Scott and Wall \cite{scott1979topological}. 
We will typically denote a graph of groups by $\GG (V, E)$, 
or simply $\GG$. 
The underlying  graph $(V, E)$ will always be connected. 
The standard involution on $E$ will be denoted 
$e\mapsto \overline{e}$ 
and the endpoint maps are $o$ and $\tau$; 
so $\tau(\overline{e})=o(e)$. 
We write $G_e=G_{\-e}$ and $G_v$ 
for the local groups at $e\in E$ and $v\in V$, 
and $\iota_e : G_e\to G_{\tau(e)}$ 
for the inclusion that is part of the data. 
We will often assume that $\GG$ is {\em reduced}, 
meaning that if $e$ is an edge with distinct vertices 
then $\iota_e$ is not surjective.

\subsection{Culler realisation}

Every element of finite order in ${\rm{Out}}(F_N)$ 
admits a geometric realisation \cite{culler1984finite}. 
More precisely, if $\phi\in{\rm{Aut}}(F_N)$ is such that  
$[\phi]$ has finite order $m$ in ${\rm{Out}}(F_N)$,
then there is a reduced graph $X$ 
with an isomorphism $\iota:\pi_1X\to F_N$ 
(involving an implicit choice of basepoint $x_0\in X$) 
and an isometry $f : X\to X$ of order $m$ 
such that $[\iota f_* \iota^{-1}] = [\phi]$, 
where $f_*\in {\rm{Aut}}(\pi_1X)$ is any representative 
of the outer automorphism determined by $f$. 
These geometric realisations 
will play a crucial role in this article. 
They are not unique in general 
-- a point that is central to the discussion in \Cref{s:last}.

\subsection{$\ell_2$-Betti numbers}

On occasion, we shall need some basic facts 
about the first $\ell_2$-Betti numbers of groups. 
In each case, the groups at hand will be 
finitely presented and residually finite, 
which means that the reader unfamiliar with these invariants 
can treat L\"{u}ck's Approximation Theorem 
\cite{luck1994approximating} as a definition: 
if $\G$ is finitely presented and residually finite
then $\beta_1^{(2)}(\G)$ is the limit 
of the numbers $\beta_1(S_n)/[\G\colon S_n]$, 
where $(S_n)$ is any nested sequence 
of finite-index subgroups in $\G$ with $\bigcap_n S_n = 1$ 
and $\beta_1(S_n)$ is the usual first Betti number, 
i.e. $\dim_\Q H_1(S_n,\Q)$. 
If $\G_1$ and $\G_2$ are finitely presented and there is 
an injection $\G_2\hookrightarrow\wh{\G}_1$ with dense image, 
then $\beta_1^{(2)}(\G_2)\ge \beta_1^{(2)}(\G_1)$ 
-- see Proposition 3.2 in \cite{bridson2013determining}. 
It follows that $\beta_1^{(2)}$ is a profinite invariant 
of finitely presented, residually finite groups.

\section{GBS Structures and Centres}
\label{sec:GBS}

In this section, we will be concerned mostly 
with finite graphs of infinite cyclic groups $\GG$, 
each equipped with a choice of generator 
for each of the local groups. 
There is a 1-1 correspondence between these structures 
and {\em integer-labelled graphs} -- i.e. 
finite graphs with an integer label at each end of each edge. 
The correspondence is given by using the indices 
$\d_e=[G_{\tau(e)} : \iota_e(G_e)]$ as labels, 
altered by a sign $\e_e=\pm 1$ 
that records whether the preferred generator of $G_e$ is sent to 
a positive or negative power 
of the preferred generator of $G_{\tau(e)}$. 
By definition, the fundamental group $\Gamma=\pi_1\GG$ 
of a finite graph of infinite cyclic groups 
is a {\em Generalised Baumslag Solitar (GBS) group}. 
If the decomposition of $\Gamma$ 
corresponds to an integer-labelled graph 
where all of the labels are positive, 
then $\Gamma$ is said to be a {\em positive} GBS group. 

An important invariant of $\Gamma$ is 
the {\em modular homomorphism} 
$\Delta_\Gamma : {\Gamma\to \mathbb{Q}^\times}$, 
which maps the local groups trivially 
and assigns to each circuit $\gamma= e_1e_2\ldots e_k$ of edges 
in the underlying graph of $\GG$ 
the rational number 
\begin{equation}
	\label{eq:fraction}
	\e_\gamma
	\frac{[G_{o(e_1)}: \iota_{\-{e}_1}(G_{e_1})] \cdot\ldots\cdot [G_{o(e_k)}: \iota_{\-{e}_k}(G_{e_k})]}%
	{[G_{\tau(e_1)}: \iota_{e_1} (G_{e_1})] \cdot\ldots\cdot [G_{\tau(e_k)}: \iota_{e_k}(G_{e_k})]}
\end{equation}
where $\e_\gamma$ is the product of the signs $\e_{e_i}$. 

The following proposition is not new; it is contained in 
Proposition 4.1 of \cite{levitt2015generalized} for example. 

\begin{proposition}
	\label{prop:graph_of_Zs}
	Let $\Gamma = F_N \rtimes_{\phi} \Z$ be 
	a free-by-cyclic group 
	such that $[\phi]$ has finite order in $\out(F_n)$. 
	Then, $\Gamma$ is a positive GBS group 
	whose modular homomorphism 
	$\Delta_\G:\G\to \mathbb{Q}^\times$ is trivial. 
\end{proposition}

\begin{proof} 
	If $[\phi]$ has finite order $m$, 
	then it can be realised as 
	an isometry of a reduced graph with fundamental group $F_n$; 
	see \cite{culler1984finite}. 
	More precisely, there is a reduced graph $X$ 
	and an isomorphism $\iota:\pi_1X\to F_n$ 
	(involving an implicit choice of basepoint $x_0\in X$) 
	and an isometry $f : X\to X$ of order $m$ 
	such that $[\iota f_* \iota^{-1}] = [\phi]$, 
	where $f_*\in {\rm{Aut}}(\pi_1X)$ is 
	any representative of the outer automorphism determined by $f$. 
	Subdividing if necessary, 
	we may assume that $f$ does not invert any edges. 
	The mapping torus of $f$ is, by definition, 
	\begin{equation*}
		M_f = X \times [0,1] \ / \ (x,0) \sim (f(x), 1). 
	\end{equation*}
	The identifications made above induce an isomorphism 
	$\pi_1 M_f \cong F_N\rtimes_\phi\Z$. 
	Observe that $M_f$ is a graph of spaces 
	where all edge and vertex spaces are circles. 
	Applying the functor $\pi_1$, 
	we obtain a decomposition of $\pi_1M_f=F_N \rtimes_{\phi} \Z$ 
	as a graph of infinite cyclic groups $\GG(V,E)$, 
	where $V$ is the set of $\<f\>$-orbits in the 0-skeleton of $X$ 
	and $E$ is the set of orbits of the 1-cells. 
	
	We choose generators of the local groups in this decomposition 
	so that the corresponding loops in $M_f$ 
	trace along $X\times (0,1)$ in the positive direction. 
	With this choice, it is clear that for each edge $e\in E$
	the inclusion $G_e\to G_{\tau(e)}$ sends 
	the preferred generator $t_e\in G_e$ to $t_{\tau(e)}^{\d_e}$, 
	where $t_{\tau(e)}$ is the preferred generator of $G_{\tau(e)}$
	and the integer $\d_e$ is the number of 1-cells of $X$ 
	that lie in the $\<f\>$-orbit $e$, 
	divided by the number of 0-cells in the $\<f\>$-orbit $\tau(e)$. 
	Note that since $\d_e>0$, 
	the labels in the corresponding integer-labelled graph 
	are all positive.
	
	If $\gamma= e_1e_2\ldots e_k$ is a circuit of edges in $(V,E)$, 
	then 
	\begin{equation} 
		\Delta_\Gamma(\gamma)=
		\frac{
			[G_{o(e_1)}: \iota_{\-{e}_1}(G_{e_1})] 
			\cdot\ldots\cdot [G_{o(e_k)}: \iota_{\-{e}_k}(G_{e_k})]
		}{
			[G_{\tau(e_1)}: \iota_{e_1} (G_{e_1})] 
			\cdot\ldots\cdot [G_{\tau(e_k)}: \iota_{e_k}(G_{e_k})]
		}
		= \frac{
			\d_{\-e_1}\dots \d_{\-e_k}
		}{
			\d_{e_1}\dots \d_{e_k}
		}.
	\end{equation} 
	The key point to observe now is that, for $i=1,\dots,k$, the fraction
	\begin{equation*}
		\frac{
			[G_{o(e_i)}: \iota_{\-{e}_i}(G_{e_i})]
		}{
			[G_{\tau(e_i)}: \iota_{e_i}(G_{e_i})]
		}
	\end{equation*} 
	is the number of $0$-cells of $X$ in the $\<f\>$-orbit $o(e_i)$ 
	divided by the number of 0-cells in the $\<f\>$-orbit $\tau(e_i)$, 
	and the latter is the same as $o(e_{i+1})$, with indices mod $k$. 
	Thus $\Delta_\Gamma(\gamma)=1$.
\end{proof}

The following corollary fails for automorphisms of surface groups, 
and this is a key point in Hempel's examples \cite{hempel2014some}. 

\begin{corollary}
\label{c:coprime-iso}
	If $[\phi]$ has order $m$ in $\outN$ and $k$ is coprime to $m$, 
	then $F_N\rtimes_\phi\Z$ and $F_N\rtimes_{\phi^k}\Z$
	are isomorphic.
\end{corollary}

\begin{proof} 
	In the proof of Proposition \ref{prop:graph_of_Zs},
	the graph $(V,E)$ and the integer labels $\delta_e$ depend only on the $\<f\>$-orbits in $X$ and not on the choice
	of generator for $\<f\>= \<f^k\>$, so $F_N\rtimes_\phi\Z$ and $F_N\rtimes_{\phi^k}\Z$ have the same GBS structure.
\end{proof}

\begin{example}
\label{e:klein}
	In the above construction, the Klein bottle group $\Z\rtimes\Z$ 
	with a non-positive GBS presentation 
	$\< x, y \mid x^{-1}yx = y^{-1} \>$ 
	emerges as the GBS associated to the labelled graph with
	two vertices and one edge, 
	with both ends of the edge labelled $2$ 
	with a positive GBS presentation $\< a, b \mid a^2 = b^2 \>$.
\end{example}

\subsection{Centres}

It is well known that 
the free-by-cyclic groups with non-trivial centre 
are precisely those considered in the previous proposition; 
we include a proof for the reader's convenience. 

\begin{lemma}
\label{free-by-cyclic-with-centre} 
	Let $\G = F_N \rtimes_\phi \Z$ be a free-by-cyclic group, 
	with $N\ge 2$, and write elements in the form $g=xt^k$ 
	with $x\in F_N$ and $\<t\> = 1\rtimes\Z$.
	
	If $[\phi]$ has infinite order in $\out(F_N)$ 
	then the centre of $\G$ is trivial.
	If $[\phi]$ has order $m$ in $\out(F_N)$ 
	then the centre of $\G$ is cyclic, 
	generated by an element of the form $xt^m$, 
	where $\ad_{x^{-1}} = \phi^m$. 
\end{lemma}

\begin{proof}   
	The centre of $\G$ injects into $\G/F_N\cong\Z$ 
	and is therefore cyclic.
	An element $xt^k$ is central if and only if 
	for any element $yt^l$ 
	\begin{equation*}
		xt^k \cdot yt^l = x \phi^k(y)t^{k+l}
		= y\phi^l(x)t^{k+l} = yt^l\cdot xt^k.
	\end{equation*}
	In particular, taking $y$ to be the identity and $l = 1$ 
	we get $\phi(x) = x$. 
	Thus, for any $y \in F_N$ we have 
	$\phi^k(y) = x^{-1}y x$, i.e. $\phi^k = \ad_{x^{-1}}$. 
	Hence the order of $\Phi = [\phi]$ in $\out(F_N)$ 
	is finite and divides $k$. 
	Conversely, if $\Phi$ has order $m$ in $\out(F_N)$, 
	then $\phi^m = \ad_{x_0}$ for some $x_0\in F_N$, 
	and calculating as above we see that $x_0^{-1}t^m$ is central.  
\end{proof}

In the following statement, 
we use the standard notation $G_e^g$ for 
the conjugate of the subgroup $G_e<\G$ by $g\in\G$. 
The careful reader may worry that
in order to make $G_e$ a bona fide subgroup of $\G=\pi_1\GG$
one needs to make a choice of maximal tree in the underlying graph. 
However, the conjugacy class of $G_e$ 
is independent of all choices, 
so the intersection in the proposition, taken over all conjugates, 
is well defined; 
it is the kernel of the action on the Bass-Serre tree. 

\begin{proposition}
\label{center_upstairs} 
	Let $\GG(V,E)$ be a reduced graph of infinite cyclic groups 
	and suppose that $\G=\pi_1\GG(V,E)$ has a non-trivial centre. 
	Then, provided $\G$ is not virtually abelian, its centre is 
	$$Z(\G) =\bigcap_{e \in E,\, g\in \G} G_e^g.$$
\end{proposition}

\begin{proof} 
	Let $z$ be a nontrivial element of $Z(\G)$. 
	There are two cases to consider: 
	in the action on the Bass-Serre tree $T$ of $\GG$, 
	either $z$ fixes a point or else 
	it acts as a hyperbolic element. 
	
	The second case is impossible: 
	as $\GG$ is reduced, the action of $\G$ on $T$ is minimal 
	(i.e. there are no proper invariant subtrees), 
	so the axis of $\<z\>$, which is $\G$-invariant, 
	would be the whole tree, contradicting the fact that 
	$\G$ is	not virtually abelian 
	and acts with cyclic arc stabilisers. 
	
	Thus each $z\in Z(\G)$ 
	fixes a non-empty subtree $T_0\subseteq T$, 
	and since this is $\G$-invariant, 
	we again use minimality to conclude $T_0=T$, 
	forcing $z$ to be contained in 
	$\bigcap \{G_e^g\mid e\in E, \, g\in \G\}$, 
	the kernel of the action on $T$. 
	
	To see that $Z(\G)$ is the whole of the kernel, 
	first observe that the latter, being non-trivial and
	an intersection of infinite cyclic groups, 
	is infinite cyclic. 
	It is also normal, so the action of $\G$ by conjugation on it 
	is either trivial or else some element acts 
	as multiplication by $-1$. 
	But this last case is impossible, 
	because the kernel contains $Z(\G)\neq 1$. 
\end{proof}

\begin{corollary}
\label{c:center_upstairs}
	Let $\G=F_N\rtimes_\phi\Z$ 
	where $[\phi]\in\outN$ has finite order, 
	considered as a positive GBS group $\G=\pi_1\GG(V,E)$ as in \cref{prop:graph_of_Zs}.
	Then, provided $\G \not\cong \Z^2$, its centre is
	$$Z(\G) =\bigcap_{e \in E,\, g\in \G} G_e^g.$$
\end{corollary}

\begin{proof}
	\Cref{free-by-cyclic-with-centre} tells us that $Z(\G) \neq 1$, 
	so the corollary is immediate 
	except in the case of the Klein bottle group, 
	which is covered by \Cref{e:klein}. 
\end{proof}

We proved \cref{prop:graph_of_Zs} by realising $\G$ 
as the fundamental group of a graph $\GG$ of infinite-cyclic groups
with a choice of generator $t_e$ for each edge group 
and $t_v$ for each vertex group 
so that the inclusions were $\iota(t_e)= t_{\t(e)}^{\d_e}$. 
With \cref{c:center_upstairs} in hand, we can describe $\G/Z(\G)$
in a similar manner, but with finite cyclic groups 
in place of the local groups in $\GG$. 
One should think of this graph of finite cyclic groups $\-{\GG}$ 
as being obtained from $\GG$ by ``reducing mod $m$'', 
where $m$ is the order of the holonomy. 
More formally, it is described as follows. 

\begin{corollary}
\label{corollary:GmodZ}
	Let $\G=\pi_1\GG(V,E)$ be as in \cref{c:center_upstairs} 
	and equip the local groups of $\GG$ 
	with generators $t_e \ (e\in E)$ and $t_v\ (v\in V)$ as above, 
	so that $\iota(t_e)= t_{\t(e)}^{\d_e}$. 
	Let $d_e = [G_e \colon Z(\G)]$, let $d_v = [G_v\colon Z(\G)]$, 
	and let $\-{\GG}(V,E)$ be the graph of groups 
	obtained from $\GG(V,E)$ by replacing 
	$G_e\ (e\in E)$ with $\-{G}_e\cong \Z/d_e$ generated by $\-t_e$ 
	and by replacing $G_v\ (v\in V)$ with $\-{G}_v\cong \Z/d_v$
	generated by $\-t_v$, defining $\iota_e(\-t_e) = \-t_v^{\d_e}$. 
	Then, 
	\begin{enumerate}
		\item $\G/Z(\G) = \pi_1\-{\GG}(V,E)$ and
		\item $\d_e = d_{\t(e)} / d_e = [G_{\t(e)}: \iota_e(G_e)] = [\-G_{\t(e)}: \iota_e(\-G_e)]$.
	\end{enumerate} 
\end{corollary}

\begin{remark}%
[Comparison with Seifert fibred spaces]
	Associated to any finite-order automorphism $f$ of a surface, 
	one has the mapping torus $M_f$, a Seifert fibred 3-manifold.
	The fundamental group of the 2-orbifold base 
	of this Seifert fibred space is $\pi_1M_f$ modulo its centre. 
	The graph of finite cyclic groups $\GG(V,E)$ 
	is the direct analogue of this base orbifold 
	in the setting of free-by-cyclic groups with finite holonomy. 
\end{remark}

\begin{lemma}
\label{l:seeFxZ} 
	Suppose $\G=F_N\rtimes_\phi\Z$ with $[\phi]\in\outN$ of finite order, with $N\ge 2$.
	Then, $\G$ is isomorphic to $F_M\times\Z$, for some $M$, 
	if and only if $\G/Z(\G)$ is torsion-free
	or, equivalently, $\Gamma/Z(\Gamma)$ is free.
\end{lemma}

\begin{proof}
	\Cref{corollary:GmodZ} tells us 
	that $\G/Z(\G)$ is the fundamental group
	of a finite graph of finite cyclic groups, 
	so it has torsion unless all of the local groups are trivial, 
	in which case it is free 
	and $\G$ is a central $\Z$-by-$F_M$ extension, 
	hence a direct product.
\end{proof}

In \Cref{s:FxZ} we shall describe 
which free-by-cyclic groups are isomorphic to $F_M\times\Z$.

\section{Reduction to free-by-(finite cyclic) case}
\label{sec:reduction}

We will work with \emph{slide moves} on graphs of groups, 
which have played a prominent role in many works on GBS groups, 
including Forester's early work \cite{forester2002deformation}. 

\begin{definition}%
[slide move]
	Let $\GG(V,E)$ be a graph of groups and ${e_1, e_2\in E}$
	edges such that $\t(e_1) = o(e_2)$. 
	A \emph{slide move} can be performed when
	$\iota_{e_1}(G_{e_1}) \leq \iota_{\bar e_2}(G_{e_2})$: 
	the move deletes $\{e_1, \-e_1\}$ from $E$ 
	and replaces it with $\{e'_1, \-e'_1\}$, 
	where $o(e_1')=o(e_1)$ and $\t(e_1') = \t(e_2)$; 
	the local group at $e_1'$ is $G_{e_1'} = G_{e_1}$ 
	with $\iota_{e_1'} = 
		\iota_{e_2} \circ \iota_{\-e_2}^{-1} \circ \iota_{e_1}$ 
	and $\iota_{\-{e}_1'} = \iota_{\-e_1}$. 
	
	This gives a new graph of groups $\GG'(V,E')$
	with $\pi_1\GG\cong \pi_1\GG'$. 
\end{definition}

Note that in the case of graphs of cyclic groups, 
${\iota_{e_1}(G_{e_1}) \leq \iota_{\bar e_2}(G_{e_2})}$ 
is equivalent to requiring that 
the index of $G_{e_1}$ in the common vertex group 
divides the index of $G_{e_2}$. 

\smallskip

\subsection{Slide moves lift}

From \cref{corollary:GmodZ} we have:

\begin{lemma}
\label{l:lifting_iso}
	Let $\G_1$ and $\G_2$ be free-by-cyclic groups 
	with non-trivial centres and positive GBS decompositions  
	$\G_i=\pi_1\GG_i$ as in \cref{prop:graph_of_Zs}. 
	If $\-{\GG}_1$ can be transformed into $\-{\GG}_2$ 
	by slide moves,  
	then $\GG_1$ can be transformed into $\GG_2$ by slide moves, 
	and hence $\G_1\cong \G_2$. 
\end{lemma}

\begin{proof} 
	The graph of finite cyclic groups $\-{\GG}(V,E')$ 
	obtained by performing a slide move on $\-{\GG}_1(V,E)$ 
	inherits from $\-{\GG}_1$ 
	a preferred set of generators for the local groups, 
	and the new edge-inclusion map $\iota_{e'}$ 
	is calculated (in terms of these generators) 
	by the indices of the edge groups involved in the slide. 
	This data lifts uniquely to define a slide move 
	from $\GG_1(V,E)$ to a positive GBS structure $\GG(V,E')$, 
	with a preferred set of generators 
	for the local groups in $\GG(V,E')$. 
	At the end of the sequence of moves, 
	the positive GBS that we obtain 
	has the same underlying graph as $\-{\GG}_2$, 
	which is the same as $\GG_2$, 
	and the inclusions of the edge groups 
	are given by the correct indices 
	$\d_e = [G_{\t(e)}: \iota_e(G_e)] 
		= [\-G_{\t(e)}: \iota_e(\-G_e)]$, 
	so this graph of infinite cyclic groups is $\GG_2$. 
\end{proof}

\subsection{Profinite equivalence passes to the central quotients}

With the equivalence of \Cref{l:lifting_iso} in hand, 
an obvious strategy emerges for deducing \Cref{t:main} 
from \Cref{thm:free-by-finite-cyclic}: 
given two free-by-(infinite cyclic) groups 
$\Gamma_1$ and $\Gamma_2$ with finite monodromy, 
prove that if $\widehat{\Gamma}_1 \cong \widehat{\Gamma}_2$ 
then the graphs of finite-cyclic groups for 
$\Gamma_1/Z(\Gamma_1)$ and $\Gamma_2/Z(\Gamma_2)$ 
are equivalent via slide moves, 
then appeal to \Cref{l:lifting_iso}. 
If this strategy is to succeed, 
$\wh{\Gamma}_i$ must determine 
the profinite completion of $\Gamma_i/Z(\Gamma_i)$, 
so we begin by proving that. 

As is standard, we shall regard $\Gamma$ 
as a subgroup of $\wh{\Gamma}$ via the canonical embedding, 
and we shall write $\overline{\Lambda}$ 
for the closure in $\wh{\Gamma}$ 
of a subgroup $\Lambda \leq \Gamma$. 

First we need a lemma which implies 
that the order of any torsion element $[\phi]\in\out(F_N)$ 
equals the order of $[\widehat{\phi}]\in\out(\widehat{F}_N)$. 

\begin{lemma}
\label{inner-after-completion}
	The canonical homomorphism $\out(F_N) \to \out(\widehat{F}_N)$ 
	induced by the completion homomorphism 
	$\aut(F_N) \to \aut(\wh{F}_N)$ 
	is injective. 
\end{lemma}

\begin{proof} 
	If $\phi\in\autN$ is not an inner automorphism of $F_N$, 
	then by Lemma~1 of \cite{grossman1974residual} 
	there exists $x\in F_N$ such that 
	$\phi(x)$ is not conjugate to $x$ in $F_N$. 
	As $F_N$ is conjugacy separable, 
	there exists a finite quotient $\pi:F_N\to Q$ 
	(extending to $\wh{\pi}:\wh{F}_N\to Q$) such that 
	$\pi(x)$ is not conjugate to $\pi\phi(x)$ in $Q$. 
	By definition, $\wh{\pi}\wh{\phi} (x) = \pi\phi(x)$ and 
	$\wh{\pi}(x) = \pi(x)$, 
	so $x$ is not conjugate to $\wh{\phi}(x) = \phi(x)$ 
	in $\wh{F}_N$, 
	and therefore $\wh{\phi}$ is not an inner automorphism. 
\end{proof}	

\begin{proposition}
\label{centers_computed}
	If $\G= F_N\rtimes_\phi\Z$ has non-trivial centre, 
	then $Z(\widehat{\G}) = \overline{Z(\G)}$
	and hence $\widehat{\G}/Z(\widehat{\G}) \cong \widehat{\G / Z(\G)}$.
\end{proposition}

\begin{proof} 
	The action of $\Z$ on $F_N$ by $n\mapsto \phi^n$ 
	has finite image in $\outN$ 
	and	therefore extends to an action of $\wh{\Z}$ 
	with the same image. 
	Thus $\widehat{F_N \rtimes_\phi \Z}
		\cong \widehat{F}_N \rtimes_{\widehat\phi} \widehat{\Z}$, 
	where the action of $\kappa\in\widehat\Z$ on $\widehat{F}_N$ is 
	$\kappa \cdot x = \widehat{\phi}^{(\kappa \mod m)}(x)$, 
	with $m$ the order of $[\phi]$ in $\outN$, 
	which by \Cref{inner-after-completion} is equal to 
	the order of $[\widehat{\phi}]$ in $\out(\widehat{F}_N)$. 
	
	In the proof of \Cref{free-by-cyclic-with-centre} 
	we calculated the centre of $F_n \rtimes_{\phi} {\Z}$ 
	and that calculation works equally well in $\widehat \G = 
		\widehat{F}_N \rtimes_{\widehat \phi} \widehat \Z$ 
	to show that 
	\begin{equation*}
		Z(\widehat \G) 
		= \{{x_0}^\kappa t^{\kappa m} \mid \kappa \in \widehat\Z\} 
		= \{(x_0t^m)^\kappa \mid \kappa \in \widehat{\Z}\},
	\end{equation*} 
	where $\< t\> = 1\rtimes\Z$ and	$\phi^m = \ad_{{x_0}^{-1}}$. 
	Thus 
	$Z(\widehat \G) = \overline{\< x_0t^m \>} = \overline{Z(\G)}$. 
\end{proof}

\subsection{Decompositions of  free-by-(finite cyclic) groups}
\label{s:free-by-fc}

By definition, a finitely generated free-by-(finite cyclic) group 
is a group $G$ that, for some $m>0$, 
fits into a short exact sequence 
\begin{equation}
\label{e:seq}
	1\to F_N \to G \to \Z/m \to 1.
\end{equation}
(This sequence need not split.)
These are precisely the groups that arise as 
fundamental groups of finite graphs of finite cyclic groups 
that are {\em orientable} in the sense that 
one can choose generators for the local groups   
so that all inclusions are described by the index, 
as in \Cref{prop:graph_of_Zs}: 

\begin{proposition}
\label{which_graphs} 
	A finitely generated group $G$ is free-by-(finite cyclic)
	if and only if there exists 
	a finite graph of finite cyclic groups $\GG(V,E)$ with 
	$G_v\cong \Z/d_v\ (e\in V)$ and 
	$G_e\cong \Z/d_e\ (e\in E)$ so that 
	\begin{enumerate}
		\item $G\cong\pi_1\GG$,
		\item $d_e$ divides $d_{\tau(e)}$ for all $e\in E$,
		\item there are generators $t_v$ for $G_v\ (v\in V)$ 
		and $t_e$ for $G_e=G_{\-e}\ (e\in E)$ so that 
		$$\iota_e(t_e) = t_{\t(e)}^{\delta_e}$$ 
		for all $e\in E$, where $\delta_e = d_{\tau(e)}/d_e$. 
	\end{enumerate} 
	Moreover, one can choose generators with these properties 
	for the local groups in any decomposition of $G$ 
	as a graph of finite groups. 
\end{proposition}

\begin{proof} 
	First assume that 
	$G$ fits into the short exact sequence~(\ref{e:seq}). 
	As $G$ is virtually free, it is the fundamental group 
	of a finite graph of finite groups $\GG$. 
	The edge and vertex groups of $\GG$ 
	intersect the torsion-free subgroup $F_N<G$ trivially, 
	so they inject into $\Z/m$ and are cyclic. 
	We fix a generator $\theta$ for $\Z/m$ and choose generators 
	for the non-trivial edge and vertex groups of $\GG$ 
	that map to $\theta^i$ with $1\le i < m$. 
	Note that the map from $G$ to $\Z/m$ 
	commutes with conjugation in $G$, 
	so this choice of generators is not subject to the ambiguity 
	inherent in regarding the local groups as subgroups of $G$ 
	(since this ambiguity is just a matter of 
	passing between conjugates). 
	
	For the converse, we assume that 
	$G$ is the fundamental group of 
	a finite graph of cyclic groups $\GG$ 
	with generators as described and 
	fix an integer $M$ 
	that is divisible by the order of each of the local groups. 
	Recall that $G=\pi_1\GG$ is generated by 
	the vertex groups $G_v=\<t_v\>$ 
	together with a set of edge symbols 
	(see page 42 of \cite{serre1980trees}). 
	We define a homomorphism $h: G\to \Z/M=\<\theta\>$ 
	by mapping the edge symbols trivially 
	and sending $t_v$ to $\theta^{\nu(v)}$ where $\nu(v)=M/|G_v|$.  
	Condition (3) ensures that $h$ is well-defined, 
	i.e. that it respects the defining relations of $G$ 
	(see page 43 of \cite{serre1980trees}). 
	Every finite subgroup of $G$ 
	is conjugate to a subgroup of a vertex group, 
	so the kernel of $h$ is	torsion-free and hence free. 
	Thus $G$ if free-by-(finite cyclic). 
\end{proof}

\begin{remark}
\label{rmk:which_come_from_G/Z}
	We saw in the previous section that 
	if $\G=F_N\rtimes\Z$ has non-trivial centre
	then $G=\G/Z(\G)$ is a free-by-(finite cyclic), 
	but not all free-by-(finite cyclic) groups 
	arise in this way. 
	One way to see this is to observe that 
	the commutator subgroup $\G'$ 
	lies in the normal subgroup $F_N \rtimes 1$, 
	and this intersects the centre $Z(\G)$ trivially 
	(provided $\G\neq \Z^2$), 
	as we saw in \Cref{free-by-cyclic-with-centre}. 
	This forces the centre of $G = \G/Z(\G)$ to be trivial: 
	if the image of $z\in\G$ in $G$ were central, 
	then $[x, z] \in Z(\G)\cap\G'=1$ for all $x \in \G$, 
	hence $z\in Z(\G)$. 
	Arguing as in \Cref{center_upstairs}, it follows that 
	the orders $|G_e|$ of the edge groups in any 
	graph-of-finite-groups decomposition of $G$ 
	must have greatest common divisor equal to $1$, 
	which is a non-trivial constraint. 
\end{remark}

\section{Solution of the reduced problem}

\subsection{
	Numbered graphs, associated posets, and slide equivalence}
\label{s:poset-slides}

In this section we introduce a key technical device -- 
the poset associated to a numbered graph, 
and prove a result that will enable us to promote isomorphisms
to sequences of sliding moves in various contexts -- 
\Cref{iso_posets_sliding_equivalent}. 

All of the information about a graph of groups 
of the type described in \Cref{which_graphs} 
is captured by the sizes of the vertex and edge groups, 
so we introduce the following data structure 
to facilitate a concise calculus of manipulation for them. 
We continue to use the word {\em graph} in the sense of Serre; 
so $\{e,\-e\}$ corresponds to a 1-cell 
in the topological realisation of the graph, 
which need not be simplicial. 

\begin{definition}
[Numbered graph] 
\label{def:numbered_graph}
	A \emph{numbered graph} $(\mathcal{G},d)$ consists of
	a finite graph $\mathcal{G}(V,E)$ 
	and a labelling $d:V\sqcup E\to \N$ such that 
	$d_e=d_{\-e}$ and $d_e$ divides $d_{\t(e)}$ for every $e\in E$. 
	A numbered graph is \emph{reduced} 
	if $d_{\t(e)} = d_e$ implies $\t(e) = \t(\-e)$ 
	(that is, $e$ is a loop).
	
	When there is no danger of ambiguity, 
	we will abbreviate $(\mathcal{G},d)$ to $\mathcal{G}$. 
\end{definition}

\begin{definition}
[Associated poset]
	To each numbered graph $(\mathcal{G},d)$ we associate 
	a finite poset $\mathsf{P}_\mathcal{G}$, 
	whose elements are the connected components 
	of certain subgraphs of $\mathcal{G}$. 
	Each element is assigned two pieces of 
	additional numerical data -- a \emph{level} and a \emph{tag}.
	\begin{itemize}
		\item 
			For $k \in \N$, let $a_{1, k}, a_{2, k}, \ldots$
			be the non-empty connected components
			of the subgraph $\mathcal{G}_k \subseteq \mathcal{G}$
			formed by the vertices and edges 
			to which $d$ assigns a value divisible by $k$.
			These $a_{i,k}$ will be the elements of 
			\emph{level} $k$.
		\item 
			$\mathsf{P}_\mathcal{G} 
				= \coprod_{k\in\N}\{a_{1, k}, a_{2, k}, \ldots\}$.
		\item 
			The ordering: define 
			$a_{i, k} \preceq a_{j, l}$
			if and only if 
			$k\ |\ l$ and $a_{i, k} \supseteq a_{j, l}$. 
		\item 
			Each element $a_{i, k}$ is \emph{tagged}
			with its first Betti number $\b_1(a_{i, k})$,
			i.e. the number of edges 
			minus the number of vertices plus one. 
	\end{itemize}
	A bijection $\mathsf{P}_\mathcal{G}\to \mathsf{P}_\mathcal{G'}$ 
	is a {\em data-preserving isomorphism} 
	if it is a poset isomorphism that preserves levels and tags. 
\end{definition}

\begin{remarks} 
	\begin{enumerate}
		\item 
			It is important to note that the labels (numbers) 
			on the edges of $\mathcal{G}$ 
			are {\em not} explicitly recorded 
			as part of the data of $\mathsf{P}_\mathcal{G}$ 
			-- cf. Step 4 in the proof of 
			\Cref{iso_posets_sliding_equivalent}. 
		\item 
			If the labels on the edges of 
			$a_{i,k} \subseteq \mathcal{G}_k$ 
			are all divisible by $k' > k$, 
			then $a_{i,k}=a_{i',k'}$ 
			as subgraphs of $\mathcal{G}_{k'}$ for some $i'$, 
			but $a_{i,k}\neq a_{i',k'}$ 
			as elements of $\mathsf{P}_\mathcal{G}$, 
			since their levels $k$ and $k'$ are not equal. 
		\item 
			The unique minimal element of $\mathsf{P}_\mathcal{G}$ 
			is $a_{1,1}=\mathcal{G}_1=\mathcal{G}$. 
			If all of the numbers in the labelling of $\mathcal{G}$ 
			are divisible by $r>1$, 
			then there will also be a unique element 
			$a_{1,r}={\mathcal{G}_r}=\mathcal{G}$ 
			at level $r$, 
			and $a_{1,1}\prec a_{1,r}$. 
	\end{enumerate}
\end{remarks} 

\begin{example}
	Let $\mathcal{G}$ be the numbered graph 
	in \Cref{fig:numbered_graph}. 
	The process of constructing 
	its associated poset $\mathsf{P}_\mathcal{G}$ 
	is portrayed in \Cref{fig:associated_poset}. 
	Note that in addition to the three elements at level $12$, 
	there is a maximal element $a_{1,4}$ at level $4$; 
	these maximal elements correspond 
	to the vertices of $\mathcal{G}$. 
	\begin{figure}
		\begin{center}
			\begin{tikzpicture}%
	[main/.style = {draw, circle, minimum size=2em, inner sep=1},
	node distance=1.5cm]
	
	\node[main] (1) {4};
	\node[main] (2) [below left of=1] {12};
	\node[main] (3) [below right of=1] {12};
	\node[main] (4) [below right of=2] {12};
	
	\draw (1) -- node[midway, sloped] {2} (2);
	\draw (1) -- node[midway, sloped] {2} (3);
	\draw (2) -- node[midway] {6} (3);
	\draw (2) -- node[midway, sloped] {3} (4);
	\draw (3) -- node[midway, sloped] {1} (4);
	\draw (1) to [out=45,in=135,looseness=5] node[midway] {4} (1);
	
	\coordinate (c1b) at ($(1)+(4,0)$);
	\node[main] (1b) at (c1b) {4};
	\node[main] (2b) [below left of=1b] {12};
	\node[main] (3b) [below right of=1b] {12};
	\node[main] (4b) [below right of=2b] {12};
	
	\draw (1b) to [out=195, in=75, looseness=1] node[midway, sloped] {2} (2b);
	\draw (1b) -- node[midway, sloped] {2} (2b);
	\draw (2b) -- node[midway] {6} (3b);
	\draw (2b) -- node[midway, sloped] {3} (4b);
	\draw (3b) -- node[midway, sloped] {1} (4b);
	\draw (1b) to [out=45, in=135, looseness=5] node[midway] {4} (1b);
\end{tikzpicture}
		\end{center}
		\caption{
			An example of a numbered graph 
			together with a graph obtained from it 
			by sliding one of the edges labelled $2$ 
			along the edge labelled $6$.}
		\label{fig:numbered_graph}
	\end{figure}
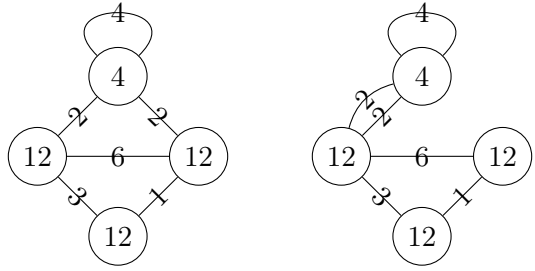
	
	\begin{figure}[p]
		\begin{center}
			\begin{tikzpicture}[
main/.style = {
	draw, circle,
	minimum size=1.8em,
	inner sep=1},
poset/.style = {
	draw, circle,
	minimum size=2.3em,
	inner sep=1},
highlight/.style = {
	line cap=round,
	line width=2.3em,
	gray, opacity=.3},
node distance=3.3em,
scale=0.47]

\coordinate (label-first) at (5,0);
\coordinate (last-poset) at (5,0);
\coordinate (poset-label) at (-1.2,0);
\coordinate (level-sep) at (0,7);
let \looseness = 200;

\def\columnNameScale{1.3};
\def\levelNameScale{1.5};
\coordinate (o-12-6-3) at (0,0.5);
\coordinate (o-4-2-1) at (0,-0.5);
\coordinate (op-4) at (0,0);
\coordinate (op-3) at (1.5,0);
\coordinate (op-2) at (3,0);
\coordinate (op-1) at (4,0);


\node (0) {};
\coordinate (l12) at (0);
\node[scale=\levelNameScale] at (l12) {$12$};

\node[main] (2) at ($(l12)+(label-first)+(o-12-6-3)$) {12};
\node (1) [above right of=2] {};
\node[main] (3) [below right of=1] {12};
\node[main] (4) [below right of=2] {12};

\draw[highlight] (2.center) -- node[left, black, opacity=.8] {$a_{1,12}$} (2.center);
\draw[highlight] (3.center) -- node[right, black, opacity=.8] {$a_{2,12}$} (3.center);
\draw[highlight] (4.center) -- node[left, black, opacity=.8] {$a_{3,12}$} (4.center);

\coordinate (p12) at ($(3)+(last-poset)-(o-12-6-3)$);
\node[poset] (a1-12) at ($(p12)+(op-3)$) {$a_{1,12}$};
\node at ($(a1-12)+(poset-label)$) {0};
\node[poset] (a2-12) [right of=a1-12] {$a_{2,12}$};
\node at ($(a2-12)+(poset-label)$) {0};				
\node[poset] (a3-12) [right of=a2-12] {$a_{3,12}$};
\node at ($(a3-12)+(poset-label)$) {0};


\coordinate (l6) at ($(l12)-(level-sep)$);
\node[scale=\levelNameScale] at (l6) {$6$};

\node[main] (6) at ($(l6)+(label-first)+(o-12-6-3)$) {12};
\node (5) [above right of=6] {};
\node[main] (7) [below right of=5] {12};
\node[main] (8) [below right of=6] {12};

\draw (6) -- node[midway] {6} (7);

\draw[highlight] (6.center) -- node[left=3em, black, opacity=.8] {$a_{1,6}$} (7.center);
\draw[highlight] (8.center) -- node[left, black, opacity=.8] {$a_{2,6}$} (8.center);

\coordinate (p6) at ($(7)+(last-poset)-(o-12-6-3)$);
\node[poset] (a1-6) at ($(p6)+(op-2)$) {$a_{1,6}$};
\node at ($(a1-6)+(poset-label)$) {0};
\node[poset] (a2-6) [right of=a1-6] {$a_{2,6}$};
\node at ($(a2-6)+(poset-label)$) {0};
\draw (a1-6) -- (a1-12);
\draw (a1-6) -- (a2-12);
\draw (a2-6) -- (a3-12);


\coordinate (l4) at ($(l6)-(level-sep)$);
\node[scale=\levelNameScale] at (l4) {$4$};

\node[main] (10) at ($(l4)+(label-first)+(o-4-2-1)$) {12};
\node[main] (9) [above right of=10]{4};
\node[main] (11) [below right of=9] {12};
\node[main] (12) [below right of=10] {12};

\draw (9) to [out=45,in=135,looseness=5] node[midway] {4} (9);

\draw[highlight] ($(9.center)+0.01*(1,1)$) to [out=45, in=135, looseness=\looseness]
node[xshift=-3.5em, yshift=-1em, black, opacity=.8] {$a_{1,4}$} ($(9.center)+0.01*(-1,1)$);
\draw[highlight] (10.center) -- node[left, black, opacity=.8] {$a_{2,4}$} (10.center);
\draw[highlight] (11.center) -- node[left, black, opacity=.8] {$a_{3,4}$} (11.center);
\draw[highlight] (12.center) -- node[left, black, opacity=.8] {$a_{4,4}$} (12.center);

\coordinate (p4) at ($(11)+(last-poset)-(o-4-2-1)$);
\node[poset] (a1-4) at ($(p4)+(op-4)$) {$a_{1,4}$};
\node at ($(a1-4)+(poset-label)$) {1};
\node[poset] (a2-4) [right of=a1-4] {$a_{2,4}$};
\node at ($(a2-4)+(poset-label)$) {0};
\node[poset] (a3-4) [right of=a2-4] {$a_{3,4}$};
\node at ($(a3-4)+(poset-label)$) {0};
\node[poset] (a4-4) [right of=a3-4] {$a_{4,4}$};
\node at ($(a4-4)+(poset-label)$) {0};

\draw (a2-4) -- (a1-12);
\draw (a3-4) -- (a2-12);
\draw (a4-4) -- (a3-12);


\coordinate (l3) at ($(l4)-(level-sep)$);
\node[scale=\levelNameScale] at (l3) {$3$};

\node[main] (14) at ($(l3)+(label-first)+(o-12-6-3)$) {12};
\node (13) [above right of=14]{};
\node[main] (15) [below right of=13] {12};
\node[main] (16) [below right of=14] {12};

\draw (14) -- node[midway] {6} (15);
\draw (14) -- node[midway, sloped] {3} (16);\

\draw[highlight, rounded corners=0.2em] (15.center) -- node[left=3em, black, opacity=.8] {$a_{1,3}$} (14.center) -- (16.center);

\coordinate (p3) at ($(15)+(last-poset)-(o-12-6-3)$);
\node[poset] (a1-3) at ($(p3)+(op-1)$) {$a_{1,3}$};
\node at ($(a1-3)+(poset-label)$) {0};

\draw (a1-3) to [out=90, in=-85] (a1-6);
\draw (a1-3) to [out=90, in=-115] (a2-6);


\coordinate (l2) at ($(l3)-(level-sep)$);
\node[scale=\levelNameScale] at (l2) {$2$};

\node[main] (18) at ($(l2)+(label-first)+(o-4-2-1)$) {12};
\node[main] (17) [above right of=18] {4};
\node[main] (19) [below right of=17] {12};
\node[main] (20) [below right of=18] {12};

\draw (17) -- node[midway, sloped] {2} (18);
\draw (17) -- node[midway, sloped] {2} (19);
\draw (18) -- node[midway] {6} (19);
\draw (17) to [out=45,in=135,looseness=5] node[midway] {4} (17);

\draw[highlight, rounded corners=0.2em]
($(17.center)+0.01*(1,1)$) to [out=45, in=135, looseness=\looseness]
node[xshift=-3.5em, yshift=-1em, black, opacity=.8] {$a_{1,2}$}
($(17.center)+0.01*(-1,1)$) --
(18.center) --
(19.center) --
(17.center);
\draw[highlight] (20.center) -- node[left, black, opacity=.8] {$a_{2,2}$} (20.center);

\coordinate (p2) at ($(19)+(last-poset)-(o-4-2-1)$);
\node[poset] (a1-2) at ($(p2)+(op-2)$) {$a_{1,2}$};
\node at ($(a1-2)+(poset-label)$) {2};
\node[poset] (a2-2) [right of=a1-2] {$a_{2,2}$};
\node at ($(a2-2)+(poset-label)$) {0};

\draw (a1-2) to [out=96.5, in=-115] (a1-6);
\draw (a2-2) to [out=85, in=-80] (a2-6);
\draw (a1-2) -- (a1-4);
\draw (a1-2) -- (a2-4);
\draw (a1-2) to [out=70, in=-85] (a3-4);
\draw (a2-2) -- (a4-4);


\coordinate (l1) at ($(l2)-(level-sep)$);
\node[scale=\levelNameScale] at (l1) {$1$};

\node[main] (22) at ($(l1)+(label-first)+(o-4-2-1)$) {12};
\node[main] (21) [above right of=22] {4};
\node[main] (23) [below right of=21] {12};
\node[main] (24) [below right of=22] {12};

\draw (21) -- node[midway, sloped] {2} (22);
\draw (21) -- node[midway, sloped] {2} (23);
\draw (22) -- node[midway] {6} (23);
\draw (22) -- node[midway, sloped] {3} (24);
\draw (23) -- node[midway, sloped] {1} (24);
\draw (21) to [out=45,in=135,looseness=5] node[midway] {4} (21);

\draw[highlight, rounded corners=0.2em]
(23.center) --
($(21.center)+0.01*(1,1)$) to [out=45, in=135, looseness=\looseness]
node[xshift=-3.5em, yshift=-1em, black, opacity=.8] {$a_{1,2}$}
($(21.center)+0.01*(-1,1)$) --
(22.center) --
(23.center) --
(24.center) --
(22.center);

\coordinate (p1) at ($(23)+(last-poset)-(o-4-2-1)$);
\node[poset] (a1-1) at ($(p1)+(op-1)$) {$a_{1,1}$};
\node at ($(a1-1)+(poset-label)$) {3};

\draw (a1-1) -- (a1-3);
\draw (a1-1) -- (a1-2);
\draw (a1-1) -- (a2-2);

\coordinate (offset) at (-1.5,0);
\coordinate (div) at ($(a4-4)-(l4)-2*(offset)$);

\coordinate (d12-up) at ($(l12)+0.5*(level-sep)$);
\draw [dashed] ($(d12-up)+(offset)$) -- ($(d12-up)+(div)$);
\coordinate (d12-6) at ($0.5*(l12)+0.5*(l6)$);
\draw [dashed] ($(d12-6)+(offset)$) -- ($(d12-6)+(div)$);
\coordinate (d6-4) at ($0.5*(l6)+0.5*(l4)$);
\draw [dashed] ($(d6-4)+(offset)$) -- ($(d6-4)+(div)$);
\coordinate (d4-3) at ($0.5*(l4)+0.5*(l3)$);
\draw [dashed] ($(d4-3)+(offset)$) -- ($(d4-3)+(div)$);
\coordinate (d3-2) at ($0.5*(l3)+0.5*(l2)$);
\draw [dashed] ($(d3-2)+(offset)$) -- ($(d3-2)+(div)$);
\coordinate (d2-1) at ($0.5*(l2)+0.5*(l1)$);
\draw [dashed] ($(d2-1)+(offset)$) -- ($(d2-1)+(div)$);

\node[scale=\columnNameScale] (level) at ($(l12)+0.65*(level-sep)$)
{\phantom{p}level $k$};
\node[scale=\columnNameScale] (components) at (level-|1)
{$k$-components\phantom{l}};
\node[scale=\columnNameScale] (poset) at (level-|a2-12)
{elements of $\mathsf{P}_\mathcal{G}$ at level $k$};

\end{tikzpicture}
		\end{center}
		\caption{
			The poset associated to 
			both of the graphs in \Cref{fig:numbered_graph}.}
		\label{fig:associated_poset}
	\end{figure}
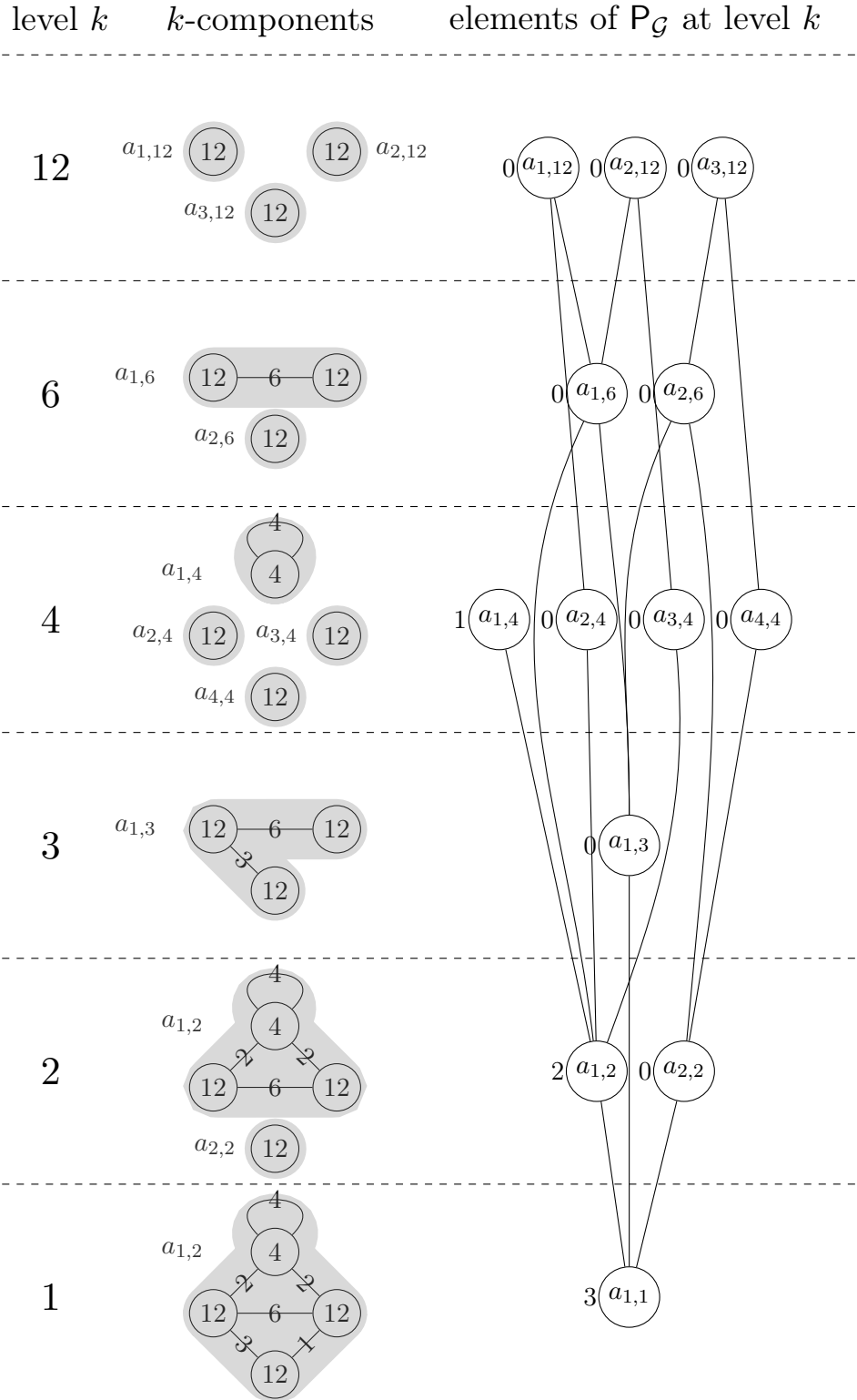
\end{example}

\begin{definition}
[Slide moves on numbered graphs]
	Let $\mathcal{G}(V,E)$ be a numbered graph 
	with distinct edges $e, f$ 
	such that $d_e$ divides $d_f$ and $\t(e) = o(f)$. 
	Consider $\mathcal{G}'=(V,E')$ 
	where $E'$ is obtained from $E$ 
	by replacing $\{e,\-e\}$ with  $\{e',\-e'\}$ 
	where $o(e') := o(e)$ and $\t(e') := \t(f)$. 
	Number the new edge by $d_{e'} := d_e$ 
	and leave the numbering of the other edges unchanged. 
	This transformation $\mathcal{G}\rightsquigarrow\mathcal{G}'$ 
	is called a {\em slide move}. 
\end{definition}

Note that slide moves take \emph{reduced} numbered graphs 
to reduced numbered graphs. 
Crucially, the definitions have been framed so that 
slide moves also induce isomorphisms of the associated posets. 

\begin{proposition}
\label{iso_posets_sliding_equivalent}
	Let $\mathcal{G}$ and $\mathcal{G}'$ 
	be reduced numbered graphs. 
	There is a data-preserving isomorphism 
	of the associated posets 
	$\mathsf{P}_\mathcal{G}$ and $\mathsf{P}_{\mathcal{G}'}$ 
	if and only if 
	$\mathcal{G}$ can be transformed into $\mathcal{G}'$ 
	by a sequence of slide moves. 
\end{proposition}

\begin{proof} 
	The definitions have been framed 
	so that the ``if'' direction is clear, 
	so we concentrate on the converse. 
	Assume that we have a data-preserving isomorphism of posets 
	$\mathsf{P}_\mathcal{G}\cong\mathsf{P}_{\mathcal{G}'}$ 
	carrying $a_{i, k}$ to $a_{i, k}'$. 
	Our proof that this isomorphism 
	can be realised by a sequence of slide moves 
	is broken into steps 
	with the following intermediate conclusions. 
	\begin{enumerate}
		\item 
		There is a  bijection $v \mapsto v'$ 
		between the vertex sets of $\mathcal{G}$ and ${\mathcal{G}'}$
		that takes the vertices of $\mathcal{G}_k$ to those of $\mathcal{G}'_k$ for each level $k$.
		\item 
		This bijection restricts to a bijection 
		between the vertex sets of each pair of corresponding components 
		$a_{i,k} \subseteq \mathcal{G}_k,\ a_{i,k}'\subseteq \mathcal{G}_k'$,
		and $a_{i,k}$ has the same number of edges as $a_{i,k}'$.
		\item 
		For each level $k$, the number of edges 
		in $\mathcal{G}_k$ and $\mathcal{G}'_k$ is equal.
		\item 
		For each $l\in \N$, the components $a_{i,k}$ and $a_{i,k}'$
		have the same number of edges labelled $l$.
		\item 
		Proceeding by top-down induction on $\mathsf{P}_\mathcal{G}$:   
		for each $a_{i,k}$, assuming that the bijection in (1) 
		extends to a label-preserving isomorphism 
		from the subgraph of $\mathcal{G}$ formed by edges with label greater than $k$
		to the corresponding subgraph of $\mathcal{G}'$, 
		we argue that 
		a sequence of slide moves on $\mathcal{G}$ 
		involving only edges in $a_{i,k}$ labelled $k$
		transforms $a_{i,k}$ into $a_{i,k}'$. 
		At the end of the induction, we have the desired transformation
		of $\mathcal{G}=a_{1,1}$ into $\mathcal{G}'=a_{1,1}'$.
	\end{enumerate}   
	\smallskip
	
	\textbf{Step 1.}  For any reduced numbered graph $\mathcal{H}$ with at least two vertices, 
	$\mathcal{H}_k$ (which is itself a reduced graph)  is disconnected for some $k$. It follows that 
	the maximal elements of $\mathsf{P}_\mathcal{G}$ correspond to vertices of $\mathcal{G}$ (some
	of the maximal elements may correspond to one-vertex subgraphs of $\mathcal{H}$ with loops, others
	will simply be vertices). Thus an isomorphism of posets $\mathsf{P}_\mathcal{G}\cong\mathsf{P}_{\mathcal{G}'}$
	(without reference to level or tags) gives a bijection of vertex sets $v\leftrightarrow v'$, and (1) follows from the
	assumption that the isomorphism is level-preserving. 
	
	\textbf{Step 2.}  
	The number of maximal elements $\succeq a_{i, k}$ in $\mathsf{P}_\mathcal{G}$ is the
	number of vertices in the subgraph $a_{i,k}\subseteq\mathcal{G}$, and similarly for $a_{i,k}'$,
	so the isomorphism $\mathsf{P}_\mathcal{G}\cong\mathsf{P}_{\mathcal{G}'}$ tells that the connected subgraphs
	$a_{i, k}$ and $a_{i, k}'$ have the same number of 
	vertices.  As the isomorphism  is data-preserving, the tags on $a_{i, k}$ and $a_{i, k}'$ are the same, and 
	these record the  first Betti numbers of $a_{i, k}$ and $a_{i, k}'$ . This proves (2). 
	
	\textbf{Step 3.} The number of edges in $\mathcal{G}_k$ is the sum of the number of edges in 
	the components $a_{i,k}\subseteq\mathcal{G}_k$, so (3) is an immediate consequence of (2).
	
	\textbf{Step 4.}
	Regarding $a_{i,k}\subseteq \mathcal{G}$ 
	as a numbered graph in its own right, 
	we have a data-preserving inclusion of posets 
	$\mathsf{P}_{a_{i,k}} \hookrightarrow \mathsf{P}_\mathcal{G}$
	whose image is spanned by the elements comparable with $a_{i,k}$.
	Thus, to prove  that $a_{i,k}$ and $a_{i,k}'$
	have the same number of edges with label $l$
	it is enough to prove the following:
	for numbered graphs $\mathcal{H}$ and $\mathcal{H}'$, if there is a data-preserving isomorphism 
	$\mathsf{P}_\mathcal{H}\cong \mathsf{P}_{\mathcal{H}'}$, then
	for each $l\in\N$ the graphs $\mathcal{H}$ and $\mathcal{H}'$
	have the same number of edges labelled $l$.
	We shall use M\"{o}bius inversion to count these edges.
	
	We require some additional notation. 
	Let  $e(\mathcal{H})$ be
	the number of edges of  $\mathcal{H}$
	and let $e_l(\mathcal{H})$ be
	the number of edges   labelled $l$.
	As before, $\mathcal{H}_m$ denotes
	the subgraph of $\mathcal{H}$
	formed by vertices and edges
	whose labels are divisible by $m$.
	Recall that the M\"{o}bius function $\mu(n)$  
	equals $0$ if $n$ is divisible by a square
	and $(-1)^{\# \text{prime factors of }n}$ otherwise.
	We will show that
	\begin{equation*}
		e_l(\mathcal{H}) = \sum_{m \st l|m} e(\mathcal{H}_m) \mu(m/l).
	\end{equation*}
	In the light of (3), this  will complete the proof.
	
	Observing that $e(\mathcal{H}_m) = \sum_{m|n} e_n(\mathcal{H}_m)$ and that
	$e_n(\mathcal{H}) = e_n(\mathcal{H}_m)$ if $m|n$,  we have  
	\begin{align*}
		\sum_{m \st l|m} e(\mathcal{H}_m) \mu(\frac{m}{l}) &
		= \sum_{m \st l|m} \quad \sum_{n \st m|n} e_n(\mathcal{H}_m) \mu(m/l)\\&
		= \sum_{m \st l|m} \quad \sum_{n \st m|n} e_n(\mathcal{H}) \mu(m/l)\\&
		= \sum_{n \st l|n} \quad \sum_{m \st l|m|n} e_n(\mathcal{H}) \mu(m/l)\\&
		= \sum_{n \st l|n} e_n(\mathcal{H}) \sum_{d \st d|\frac{n}{l}} \mu(d)\\&
		= \sum_{n \st l|n} e_n(\mathcal{H}) \cdot \mathbf{1}\big(\frac{n}{l} = 1\big)\\&
		= e_l(\mathcal{H}).
	\end{align*}
	
	\textbf{Step 5.} At this stage in the proof, we can assume that $a_{i, k}$ and $a_{i, k}'$ 
	are numbered graphs on the same set of vertices, 
	with  the same set of edges carrying  a label greater than $k$
	and with the same number of edges of label $k$. We want to perform  slide moves
	involving only edges labelled $k$ to transform $a_{i,k}\subseteq \mathcal{G}_k$ into $a_{i,k}'\subseteq \mathcal{G}_k'$,
	leaving the remainder of $\mathcal{G}$ unchanged.
	If we can do this, then we can inductively transform $\mathcal{G}$ into $\mathcal{G}'$
	by slide moves, arguing with an induction that assumes we have already adjusted $\mathcal{G}$
	so that $\mathcal{G}$ into $\mathcal{G}'$ have the same set of edges numbered with higher labels.
	
	For the base step of the induction, no slide moves are required: each
	maximal element $a_{i,k}\in \mathsf{P}_\mathcal{G}$ is a component with just one vertex and it  may 
	have some loops whose label is $k$, the level of the vertex; 
	the number of these loops is the first Betti number, which is recorded as the tag on $a_{i,k}$
	and hence is the same for corresponding maximal element $a_{i,k}'$ of $ \mathsf{P}_{\mathcal{G}'}$.  
	
	For the inductive step, we first reduce to the case where
	\emph{all} edges of $a_{i, k}$ and $a_{i, k}'$ are labelled $k$. We do this by observing that
	within each connected component $C$ of the subgraph of $a_{i,k}$
	that would be obtained  by deleting all edges labelled $k$,
	we can freely slide the endpoints of each edge labelled $k$.
	Moreover, by induction, each component $C$ in $a_{i,k}$  is identified with an identical component $C'$ in $a_{i,k}'$.
	So we can choose a vertex $v_c\in C$ and assume that whenever an edge of $a_{i,k}$ or $a_{i,k}'$
	labelled $k$ has an endpoint in $C$, that endpoint is $v_C$. We can then forget the remainder of $C$ (it will
	remain inertly attached to $v_C$ while we perform slide moves on the connected subgraph where all edges are labelled $k$).
	
	Finally, we observe that a numbered graph in which all edges are labelled $k$ can be transformed by slide moves
	into a  simple  normal form, namely the {\em octopus graph} shown in figure \ref{fig:octopus}. 
	In more detail,  we fix a vertex $v_0$ and note that for
	every edge $e$, at least one of $o(e)$ and $t(e)$
	is connected to $v_0$ by a path which doesn't include the edge $e$;
	we  slide $e$ along the path until one of its endpoints is at $v_0$; proceeding one edge at a time, this
	operation transforms our numbered graph into a graph where  every edge has $v_0$
	as at least one of its endpoints and the vertices other than $v_0$ are leaves (vertices of valence 1). 
	To complete the proof, we  transform each of $a_{i, k}$ and $a_{i, k}'$ into an octopus in this way, 
	and note  that the octopus is determined by its  number of vertices and edges, which is the same as the
	number in $a_{i, k}$ and $a_{i, k}'$.
	\begin{figure}
		\begin{center}
			\begin{tikzpicture}
[line cap=round,
line join=round,
x=1cm,
y=1cm]

\def\lineWidth{1.5pt}
\def\circleSize{2pt}

\draw [line width=\lineWidth] (0,0)-- (-1.73,-1);
\draw [line width=\lineWidth] (0,0)-- (-1,-1.73);
\draw [line width=\lineWidth] (0,0)-- (0,-2);
\draw [line width=\lineWidth] (0,0)-- (+1,-1.73);
\draw [line width=\lineWidth] (0,0)-- (+1.73,-1);
\draw [line width=\lineWidth] (0,0) .. controls (-2,0) and (0,2) .. (0,0);
\draw (0.5,0) node {$v_0$};
\draw [line width=\lineWidth] (0,0) .. controls (-1.5,1.5) and (1.5,1.5) .. (0,0);
\draw [fill=black] (0,0) circle (\circleSize);
\draw [fill=black] (-1.73,-1) circle (\circleSize);
\draw [fill=black] (-1,-1.73) circle (\circleSize);
\draw [fill=black] (0,-2) circle (\circleSize);
\draw [fill=black] (+1,-1.73) circle (\circleSize);
\draw [fill=black] (+1.73,-1) circle (\circleSize);
\end{tikzpicture}
		\end{center}
		\caption{An octopus graph.}
		\label{fig:octopus}
	\end{figure}
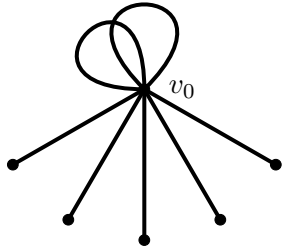
\end{proof}

We close this section with a reminder about why we are interested in numbered graphs and slide moves on them.

\begin{lemma}\label{l:sliding-the-same}
	There is a 1-1 correspondence  $\GG \leftrightarrow \mc{G}$
	between the finite graphs of finite groups described in Proposition \ref{which_graphs} and
	the set of finite numbered graphs. If $\GG_i$ corresponds to $\mc{G}_i$ for $i=1,2$, then
	there is a slide move transforming $\GG_1$ into $\GG_2$ if and only if there is a slide move
	transforming $\mc{G}_1$ into $\mc{G}_2$.
\end{lemma}

\begin{proof}
	Given a graph $\GG$ of finite cyclic groups, the corresponding numbered graph is obtained by replacing the
	local groups with their order. One can recover  $\GG$ from the numbered graph by taking edge and vertex groups to be
	cyclic of the prescribed order, with preferred generators $t_e$ and $t_v$, defining $\iota_e(t_e) = t_{\t(e)}^{\delta_e}$
	where $\delta_e$ is the label on the vertex $t(e)$ divided by the label on $e$. It is obvious from the definitions that
	this correspondence respects slide moves.
\end{proof}

\subsection{Finite subgroups, centralisers, and $\tsc(G)$}
\label{sec:torsion}

In this section, we translate the purely combinatorial arguments 
of \Cref{s:poset-slides} into group theory. 

\begin{definition}
[The {\bf{f}}inite {\bf{s}}ubgroups {\bf{c}}onjugacy poset]
\label{def:torsion-subgroups-conjugacy-poset}
	Let $G$ be a group. 
	The elements of the 
	\emph{finite subgroups conjugacy poset} $\tsc(G)$ 
	are the conjugacy classes $[S]$ of finite subgroups $S \leq G$, 
	with the ordering $[S_1] \preceq [S_2]$ 
	if a conjugate of $S_1$ is contained in $S_2$. 
	
	Each element of $\tsc(G)$ is assigned two pieces 
	of additional numerical data: 
	its {\em level}, which is the cardinality $|S|$,
	and its {\em tag}, 
	which is the first Betti number of the centraliser $C_G(T)$. 
	By definition, a {\em data-preserving isomorphism} 
	$\tsc(G_1)\to \tsc(G_2)$ 
	is a poset isomorphism that preserves this numerical data. 
\end{definition}

\begin{proposition}
\label{associated_poset_is_torsion_subgroups}
	Let $G$ be a finitely generated free-by-(finite cyclic) group, 
	let $\GG$ be a reduced graph of finite groups 
	with $G\cong \pi_1\GG$, 
	and let $\mc{G}$ be the numbered graph that encodes $\GG$ 
	(as in \Cref{l:sliding-the-same}).  
	
	Then, there is a data-preserving isomorphism 
	$\Psi: \tsc(G)\to \mathsf{P}_\mathcal{G}$. 
\end{proposition}

\begin{proof}
	Consider the action of $G=\pi_1\GG$ on
	the universal covering $T$ of $\GG$, which is the Bass-Serre tree of the splitting.
	The fixed-point set ${\rm{Fix}}(S)\subseteq T$ of each finite subgroup $S<G$
	is connected and non-empty, and its image in $\mc{G}=G\backslash T$, which we call  $X_S$, is the same as the image of
	${\rm{Fix}}(S^g)$ for every conjugate $S^g=g^{-1}Sg$. All of the edges and vertices in $X_S$
	have a copy of $S$ in the local group, and hence the numbers labelling the edges and vertices of $X_S\subseteq\mc{G}$
	are all divisible by $k=|S|$; in other words $X_S\subseteq\mc{G}_k$. In fact,
	$X_S$ is a connected component of $\mc{G}_k$: for each vertex $v\in X_S$ and each
	edge $e$ with $\tau(e)=v$, if $e$ is in $\mc{G}_k$ then $k$ divides
	the label on $e$, which means that  $G_e<G_v$ contains a subgroup of cardinality $k$, and the only such subgroup in  
	$G_v$ (which is cyclic) is $S$, hence $S \leq G_e$ and $e\subset X_S$.  Thus $X_S=X_{S^g}$  is one of the elements
	of level $|S|$ in  $\mathsf{P}_\mathcal{G}$ and we can define $\Psi([S]) := X_S$.
	Note that $\Psi$ is injective, because each of the local groups  in $\GG$ contains a unique subgroup
	of any given cardinality, so $X_{S_1}$ and $X_{S_2}$ are disjoint if $S_1$ and $S_2$ are non-conjugate subgroups
	of the same size.
	
	To see that $\Psi$ is surjective, consider an element $a_{i,k}\in \mathsf{P}_\mathcal{G}$. By definition,
	$a_{i,k}$ is a subgraph of $\mc{G}_k$, so the local group at each vertex $v$ of $a_{i,k}$ contains a subgroup of
	cardinality $k$. If this subgroup is $S$, then $v\in X_S \cap a_{i,k}$. 
	But $a_{i,k}$ and $X_S$ are connected components of $\mc{G}_k$, so having non-trivial intersection forces $a_{i,k}=X_S$,
	in other words $a_{i,k} = \Psi([S])$. Thus $\Psi$ is a bijection.
	
	If $S_1 \leq S_2$, then ${\rm{Fix}}(S_1)\supseteq {\rm{Fix}}(S_2)$ 
	and hence $X_{S_1}\supseteq X_{S_2}$.
	Moreover, $|S_1|$ divides $|S_2|$. 
	Thus $S_1 \leq S_2$ implies $\Psi([S_1]) \preceq \Psi([S_2])$,
	so $\Psi$ is an isomorphism of posets.
	
	It remains to prove that $\Psi$ preserves tags, i.e. that the first Betti number of $C_G(S)$ is equal to the
	first Betti number of $X_S = \Psi([S])$. For this, the first thing to observe is that 
	while $N_G(S)$ preserves the fixed-point set ${\rm{Fix}}(S)$ in the Bass-Serre tree $T$,
	the elements of $G$ that are not in $N_G(S)$ move ${\rm{Fix}}(S)$  off itself entirely. 
	Indeed,  if $v$ and $gv$ are both vertices of ${\rm{Fix}}(S)$, then $S$ and $S^g$ are both subgroups of the
	stabiliser of $v$, which being cyclic has only one subgroup of size $|S|$, so $S=S^g$ and $g\in N_G(S)$.
	It follows that $N_G(S)\backslash{\rm{Fix}}(S)$ injects into $G\backslash T = \mathcal{G}$ with image $X_S$.
	Thus $N_G(S)$ is the fundamental group of a graph of finite groups with underlying graph $X_S$,
	and therefore $N_G(S)$ has  the same first Betti number as $X_S$. 
	The final thing to observe is  $N_G(S) = C_G(S)$, because in any (torsion-free)-by-abelian group,
	the normaliser of a finite subgroup  is the centraliser of that subgroup. 
\end{proof}

\begin{proof}
[Proof of Theorem \ref{thm:same_posets_implies_iso}]
	Let $G_1$ and $G_2$ be finitely generated free-by-(finite cyclic) groups. We must show that
	if there is a data-preserving isomorphism $\tsc(G_1)\cong\tsc(G_2)$ then  $G_1\cong G_2$. 
	To this end, we decompose each $G_i$ as a graph of finite cyclic groups $\GG_i$, as in Proposition \ref{which_graphs},
	and  encode $\GG_i$ as a numbered graph $\mc{G}_i$, as in Lemma \ref{l:sliding-the-same}.
	Proposition \ref{associated_poset_is_torsion_subgroups}
	tells us that there is a data-preserving isomorphism of the associated posets $\mathsf{P}_{\mathcal{G}_1} \cong 
	\mathsf{P}_{\mathcal{G}_2}$;  \Cref{iso_posets_sliding_equivalent} tells us that this isomorphism can be
	realised by a sequence of slide moves transforming  $\mathcal{G}_1$  into $\mathcal{G}_2$; and Lemma \ref{l:sliding-the-same}
	allows us to deduce that $\GG_1$ can be transformed into $\GG_2$ by slide moves. As slide moves on graphs of
	groups preserve fundamental groups, we conclude that $G_1\cong G_2$, as desired. 
\end{proof}  

\subsection{Completions detect finite subgroups and centralisers}
\label{sec:completions_detect}

Several of our previous proofs have exploited the fact that Bass-Serre theory allows one to capture the
behaviour of finite subgroups of virtually free groups $G$. In this section, we shall exploit the fact that no additional
torsion appears when one passes from $G$ to  its profinite completion $\wh{G}$; this is not true of finitely generated, residually finite groups in general. We continue to regard $G$ as a subgroup of $\wh{G}$ and write $\overline{H}$ for the closure of a
subgroup $H \leq \wh{G}$.

\begin{proposition}
\label{torsion}
	If a finitely generated group $G$ is virtually free, 
	then every finite subgroup of $\widehat G$ is conjugate to a finite subgroup of $G$,
	and if two finite subgroups of $G$ are conjugate in $\widehat G$, they are already conjugate in $G$.
\end{proposition}

\begin{proof}
	Let $\mathcal{G}$ be a finite graph of finite groups
	such that $G$ is its (discrete) fundamental group. Proposition 8.2.3 in
	\cite{ribes2017profinite} shows that
	$\widehat G$ is actually isomorphic 
	to the \emph{profinite} fundamental group of $\mathcal{G}$,
	which means that there is a profinite tree $\widehat{T}$ on which $\widehat{G}$ acts 
	with quotient $\mathcal{G}$.
	
	As in the discrete setting, every  finite subgroup $S$ of $\widehat G$
	will fix some vertex  of $\widehat{T}$, and therefore $S$ is conjugate in $\widehat G$
	to a subgroup of one of the vertex groups of $\mathcal{G}$, which all lie in $G$. 
	
	The assertion about non-conjugate finite subgroups of $G$ remaining non-conjugate in $\widehat G$
	is a special case of  conjugacy separability for subgroups of  virtually-free groups, which was
	proved by Chagas and Zalesskii   \cite{chagas2015subgroup}. 
\end{proof}

\begin{proposition}
\label{centralisers}
	Let $G$ be a finitely generated free-by-(finite cyclic) group and $S < G$ be a finite subgroup.
	Then $C_{\widehat{G}}(S) \cong \widehat{C_G(S)}$.
\end{proposition}

\begin{proof}
	Ribes--Zalesskii in Theorem 2.6 of \cite{ribes2014normalizers}
	proved that for virtually-free groups 
	$N_{\widehat G}(S) = \overline{N_G(S)}$. 
	And since both $G$ and $\wh{G}$ are 
	(torsion-free)-by-abelian group, 
	$N_{\widehat G}(S) = C_{\widehat{G}}(S)$ and $N_G(S)=C_G(S)$. 
	
	In hyperbolic groups, the centralisers of finite subgroups 
	are always finitely generated, 
	and in virtually free groups all finitely generated subgroups 
	are closed in the profinite topology (the LERF property), 
	so $\overline{C_G(T)} \cong \widehat{C_G(T)}$. 
\end{proof}

\subsection{Proof of the main theorems} 
\label{sec:proof-of-main-thms}

We recall the statements of our main theorems, 
for the convenience of the reader. 

\thmfinitecyclic*

\begin{proof}
	With \Cref{thm:same_posets_implies_iso} in hand, 
	it suffices to show that if $\wh{G}_1\cong\wh{G}_2$ 
	then there is a data-preserving isomorphism of posets 
	$\tsc(G_1)\cong\tsc(G_2)$. 
	
	To this end, we fix an identification $\wh{G}:=\wh{G}_1=\wh{G}_2$ 
	and consider the canonical embeddings 
	$G_1\hookrightarrow\wh{G}$ and $G_2\hookrightarrow\wh{G}$. 
	\Cref{torsion} tells us that 
	for each conjugacy class of finite subgroups $S_1 \leq G_1$ 
	there is a unique conjugacy class 
	of finite subgroups $S_2 \leq G_2$ 
	such that $S_1$ is conjugate to $S_2$ in $\wh{G}$. 
	Thus we obtain an isomorphism $\tsc(G_1)\cong \tsc(G_2)$ 
	and we will done if we can argue 
	that the first Betti numbers 
	of $C_{G_1}(S_1)$ and $C_{G_2}(S_2)$ are the same. 
	
	As $S_1$ is conjugate to $S_2$ in $\wh{G}$, 
	they have isomorphic centralisers in $\wh{G}$, 
	so \Cref{centralisers} tells us that 
	$C_{G_1}(S_1)$ and $C_{G_2}(S_2)$ have 
	the same profinite completion. 
	It follows that $C_{G_1}(S_1)$ and $C_{G_2}(S_2)$ 
	have the same first Betti number, because 
	the first Betti number of a group $G$ is encoded in its finite quotients: it is the largest
	integer $b$ such that $G$ maps onto $(\Z/p)^b$ for every prime $p$. 
\end{proof}

\begin{lemma}
\label{l:l2centre}
	Let $\G=F_{N}\rtimes_{\phi}\Z$, 
	where $N\ge 2$ and $[\phi]\in{\rm{Out}}(F_N)$ has finite order. 
	Let $\Lambda$ be a finitely presented, residually finite group. 
	If $\wh{\G}\cong \wh{\Lambda}$, 
	then the centre of $\Lambda$ is infinite. 
\end{lemma}

\begin{proof} 
	From \Cref{centers_computed} we know that 
	$Z(\wh{\G}) \cong\wh{\Z}$ is torsion-free, 
	so if $Z(\Lambda)$ were not infinite, 
	then $\Lambda$ would inject into $\wh{\G}/Z(\wh{\G})$ 
	as a dense subgroup. 
	\Cref{centers_computed} also tells us that  
	$\wh{\G}/Z(\wh{\G})$ is isomorphic 
	to the profinite completion of $\G/Z(\G)$. 
	And since contains $\wh{F}_N$ as a subgroup of finite index, 
	any dense finitely presented subgroup 
	must have positive first $\ell_2$ Betti number, 
	see Lemma 3.2 in \cite{bridson2021profinite}. 
	But $\beta_1^{(2)}(\Lambda)=0$, 
	because $\wh{\Lambda}\cong\wh{\G}$ implies 
	$\beta_1^{(2)}(\Lambda)=\beta_1^{(2)}(\G)$
	and $\beta_1^{(2)}(\G)=0$ 
	since $\G$ contains a finitely generated 
	infinite normal subgroup of infinite index \cite{gaboriau}.  
\end{proof}

\thmmain*

\begin{proof}
	Lemma \ref{l:l2centre} tells that $Z(\G_2)$ is non-trivial, 
	so $[\phi_2]$ has finite order, 
	by \Cref{free-by-cyclic-with-centre}. 
	We consider $\G_1$ and $\G_2$ as GBS groups $\G_i=\pi_1\GG_i$, 
	as in \Cref{prop:graph_of_Zs}, 
	and we decompose each of the 
	free-by-(finite cyclic groups) $G_i:=\G_i/Z(\G_i)$ 
	as a graph of finite cyclic groups $G_i=\pi_1\-{\GG}_i$ 
	as in \Cref{corollary:GmodZ}. 
	The decompositions $\-{\GG}_i$ are encoded 
	by numbered graphs $\mathcal{G}_i$, 
	which we studied in \Cref{s:poset-slides}. 
	
	From \Cref{centers_computed} we know that $\wh{\G}_1\cong\wh{\G}_2$ implies
	$\wh{G}_1 \cong \wh{G}_2$, so from \Cref{thm:free-by-finite-cyclic} we conclude that $G_1\cong G_2$. 
	In particular, there is a data-preserving poset isomorphism 
	$\tsc(G_1)\cong\tsc(G_2)$.
	\Cref{associated_poset_is_torsion_subgroups} translates this 
	into a data-preserving isomorphism 
	$\mathsf{P}_{\mathcal{G}_1}\cong\mathsf{P}_{\mathcal{G}_2}$ 
	between the posets associated 
	to the numbered graphs ${\mathcal{G}_i}$. 
	Then, from \Cref{iso_posets_sliding_equivalent} we know that 
	$\mathcal{G}_1$ can be transformed into $\mathcal{G}_2$ 
	by a sequence of slide moves,
	and from \Cref{l:sliding-the-same} we deduce that 
	$\-{\GG}_1$ can be transformed into $\-{\GG}_2$ 
	by a sequence of slide moves. 
	\Cref{l:lifting_iso} allows us to lift this 
	last sequence of slide moves 
	so as to transform $\GG_1$ into $\GG_2$, 
	from which we conclude $\G_1 \cong \G_2$, as claimed. 
\end{proof}

\begin{remark}
[On the profinite rigidity of GBS groups] 
\label{r:GBS}
	The only GBS groups that are residually finite are those that have an infinite cyclic normal subgroup (corresponding to the modular homomorphism having image in $\{-1,1\}$)
	and those that are solvable \cite{levitt2015generalized}. 
	The latter are the Baumslag-Solitar groups ${\rm{BS}}(1,n)$, which can be distinguished from each other by their finite abelian quotients. Moreover, the solvable examples are distinguished from the non-solvable ones by the fact that their profinite completions are solvable. 
	
	If $\G$ is a GBS with an infinite cyclic normal subgroup, then either this subgroup
	is central and $\G$ has the form $F_N\rtimes_\phi\Z$, or else there is a subgroup of index 2 that has this form.
	\Cref{l:l2centre} tells us 
	that the profinite completion of $\G$ 
	distinguishes between these two cases -- 
	so in the light of \Cref{t:main}, 
	we obtain the analogue of \Cref{t:1-relator}. 
	At this point, we would like to claim that 
	if $\G_1$ and $\G_2$ are residually finite GBS groups 
	and $\wh{\G}_1\cong\wh{\G}_2$, then $\G_1\cong\G_2$. 
	The remaining case is where both $\G_1$ and $\G_2$ have trivial
	centre but they have a common subgroup of index $2$ that has infinite centre. 
	This requires additional work. 
\end{remark}

\section{One-relator groups}
\label{s:1-relator}

\thmonerelator*

\begin{proof} 
	Baumslag and Taylor \cite{baumslag1967centre} 
	proved that 1-relator groups with non-trivial centre 
	are free-by-cyclic $F_N\rtimes\Z$, 
	so in the light of \Cref{t:main} 
	it suffices to argue that $Z(\G_2)\neq 1$. 
	\Cref{l:l2centre} assures us that this is the case. 
\end{proof}

The following result shows that 
both \Cref{t:1-relator} and \Cref{l:l2centre} would fail 
if we dropped the assumption that $\G_2$ is residually finite. This was already observed by Baumslag, Miller and Troeger \cite{BaumMillerTroeger}, but we include a short proof for the reader's convenience.

\begin{theorem} [\cite{BaumMillerTroeger}]
	For every 1-relator group $\G= \<A\mid r\>$, with $|A|\ge 2$ and $r\neq 1$, there exists a 1-relator group  $\G^\dagger = \< A\mid \tilde{r}\>$ that is not residually finite but admits an epimorphism $\pi: \G^\dagger \to \G$ such that
	\begin{enumerate}
		\item $\G^\dagger$ is not residually finite;
		\item $\ker\pi\neq 1$;
		\item $\wh{\pi}:\wh{\G^\dagger}\cong\wh{\G}$ is an isomorphism.
	\end{enumerate} 
	In particular, there is a 1-relator group that is not residually finite but has the same finite quotients as $\G$. 
\end{theorem}

\begin{proof} Assume that $r$ is cyclically reduced.
	Fix $t\in A$ so that $r$ is not a power of $t$, let $\tilde{r} = (t^{-1}rt)^{-1}r(t^{-1}rt)r^{-2}=[t^{-1}rt,\, r]r^{-1}$
	and let $\pi:  \G^\dagger \to \G$ be the epimorphism induced by the identity map on $A$.
	
	The following variation on an argument of Baumslag \cite{baumslag1969non} 
	originates in the work of Higman \cite{higman1951finitely}. 
	In any finite quotient $Q$ of $\G^\dagger$, if the image of $r$ has order $d$ then
	in $Q$ we have $r=(t^{-1}r^dt)^{-1}r(t^{-1}r^dt)=r^{2^d}$, hence $d$ divides $2^d-1$, which is impossible if $d>1$ because the smallest prime factor of $d$
	is less than the smallest prime factor of $2^d-1$. Thus $r$ must have trivial image in any finite quotient of $\G^\dagger$, and
	therefore $\pi:\G^\dagger \to \G^\dagger/\<\! \<r\>\!\>= \G$ induces an isomorphism of profinite completions.
	
	To prove (1) and (2), we must argue that  $r\neq 1$ in $\G^\dagger$. If $r$ were equal to $1\in \G^\dagger$, then
	$\pi$ would be an isomorphism, induced by the identity map on $A$, and hence the normal closures of $r$ and $\tilde{r}$ in the free group $F(A)$ would be equal. 
	But a classical result of Magnus \cite{magnus} assures that this is not true, because $\tilde{r}$ is not conjugate to $r^{\pm 1}$ in $F(A)$.
\end{proof}

\section{Isomorphisms of free-by-cyclic groups}
\label{s:last}
\label{s:isomorphism}

In the introduction to this paper, we promised to give 
examples illustrating various ways in which non-obvious isomorphisms 
between free-by-cyclic groups with centre can arise. 
In this section, we will fulfil these promises as follows:

\smallskip
\begin{itemize}
	\item 
		We explain why, for each $M\ge 2$, 
		there is an automorphism $\phi\in\autN$ such that 
		$F_N\rtimes_{\phi}\Z\cong F_M\times\Z$ 
		if and only if $M-1$ divides $N-1$. 
		In each	case, 
		the conjugacy class of $[\phi]\in \outN$ is unique. 
	\item
		We describe elements $\Phi\in\outN$ of finite order $m$ 
		with some powers $\Phi^j$, where $j$ is coprime to $m$, 
		so that $\Phi$ is not conjugate to $\Phi^j$ in $\outN$ 
		even though $F_N\rtimes_\phi\Z\cong F_N\rtimes_{\phi^j}\Z$ 
		where $\Phi = [\phi]$.  
	\item
		We also describe elements $\Phi,\Psi\in\outN$ 
		of finite order so that 
		$\<\Phi\>$ is not conjugate to $\<\Psi\>$ in $\outN$ 
		but $F_N\rtimes_\phi\Z\cong F_N\rtimes_{\psi}\Z$ 
		for $\Phi = [\phi]$ and $\Psi = [\psi]$.
\end{itemize}
\smallskip

Note that in each case 
the isomorphisms between the semidirect products 
cannot preserve the free-group fibre. These phenomena are
known to experts, e.g.~Khramtsov \cite{khramtsov} and Cashen--Levitt \cite{cashenlevitt}. 
This topic is also being explored in an ongoing work of Andrew and Patchkoria, 
who obtained a version of \Cref{p:alg-version} independently \cite{AP}. 

\subsection{The disguises of $F_M\times\Z$}
\label{s:FxZ}

We know from \Cref{t:main} that 
if $\G=F_N\rtimes_\phi\Z$ has the same finite quotients as $F_M\times\Z$
then $\G\cong F_M\times\Z$. 
In fact, this can be proved much more directly: 
$Z(\G)$ is infinite, by \Cref{l:l2centre}, 
and \Cref{centers_computed} tells us that $\G/Z(\G)$ is torsion-free, 
since it embeds in $\wh{\G/Z(\G)}\cong \wh{F}_M$, so
\Cref{l:seeFxZ} completes the proof. 
In this section, we investigate what $N$ and $\phi$ can be.  

\begin{proposition}
\label{p:alg-version}
	For all integers $M\ge 2$ and $N\ge M$, 
	if $M-1$ divides $N-1$ 
	then there is a unique conjugacy class of elements $[\phi]\in\outN$ 
	such that $F_N\rtimes_{\phi}\Z\cong F_M\times\Z$.
	
	If $M-1$ does not divide $N-1$, 
	then no group of the form $F_N\rtimes\Z$ 
	is isomorphic to $F_M\times\Z$. 
\end{proposition}

\begin{proof} 
	As in the proof of \Cref{prop:graph_of_Zs}, 
	if $\Phi = [\phi]\in\outN$ has finite order $m$, 
	we can realise it by an isometry $f:X\to X$ 
	of a reduced graph with $\pi_1X\cong F_N$ 
	and thereby express $\G=F_N\rtimes_{\phi}\Z$ 
	as the fundamental group 
	of a reduced finite graph of infinite cyclic groups, 
	and express $\G/Z(\G)$ as the fundamental group 
	of a graph of finite-cyclic groups, as in \Cref{corollary:GmodZ}. 
	In order for $\G/Z(\G)$ to be torsion-free, 
	as required by \Cref{l:seeFxZ}, 
	it is necessary and sufficient
	that the local groups in this decomposition be trivial, 
	which means that 
	all of the integers $d_e$ and $d_v$ in \Cref{corollary:GmodZ} 
	must equal $1$ and the integers $\delta_e$ 
	in the proof of \Cref{prop:graph_of_Zs} must equal $1$. 
	This means all of the $\<f\>$-orbits of cells in $X$ 
	must have the same size, 
	and since $m$ is the order of $f$, 
	$\<f\>\cong\Z/m$ must act freely. 
	Thus $\G=F_N\rtimes_{\phi}\Z$ will be 
	isomorphic to a group of the form $F_M\times\Z$ if and only if 
	the action of $\Z/m\cong\<\Phi\>\cong\< f\>$ 
	in any reduced geometric realisation $f:X\to X$ is free. 
	(In fact, the argument shows that 
	if this is true in one reduced geometric realisation, 
	then it is true for them all.) 
	In this case, $\G/Z(\G)\cong F_M$ will be 
	the fundamental group of the graph $X/\<f\>$. 
	Thus this quotient must have Euler characteristic $1-M$ 
	and its $m$-sheeted covering $X$ 
	must have Euler characteristic $m(1-M)$. 
	But $\pi_1X\cong F_N$, so $\chi(X) = 1-N = m(1-M)$. 
	
	At this stage, we have proved the assertion 
	in the second sentence of the proposition 
	and we have proved that when $M-1$ does divide $N-1$, 
	if $F_N\rtimes_{\phi}\Z\cong F_M\times\Z$ 
	then $\Phi\in\outN$ must have order $m=(N-1)/(M-1)$. 
	It remains to prove that such elements $\phi$ exist 
	and that $\Phi$ is unique up to conjugacy in $\outN$. 
	
	We have argued that the existence of $\phi$ 
	is equivalent to the existence of a free action of $\Z/m$ 
	on a reduced graph of Euler characteristic $1-N$. 
	To obtain such an action, 
	we fix an identification of $F_M$ 
	with the fundamental group of the $M$-rose $R_M$ 
	(the 1-vertex graph with $M$ loops), 
	take an epimorphism ${p:F_M\to \Z/m}$, 
	choose a primitive element $a\in F_M$ 
	whose image generates $\Z/m$ 
	and consider the covering space $X\to R_M$ 
	corresponding to $\ker p<\pi_1 R_M$. 
	The Euler characteristic of this cover is $1-N$, 
	so its fundamental group is free of rank $N$. 
	Conjugation by $a$ induces 
	a deck transformation $f : X\to X$ 
	generating the Galois group, which is cyclic of order $m$; 
	this group acts freely by isometries. 
	We fix an isomorphism $\iota:\pi_1(X,x_0)\cong F_N$ 
	and take $\phi$ to be any automorphism 
	that lies in the outer automorphism class defined by $f$ 
	(transported by $\iota$). 
	Note that since $f$ does not fix the basepoint $x_0\in X$, it is only the outer automorphism class of $\phi$ that is defined
	by $\iota$ and $f$. Note also that if we change the marking $\iota:\pi_1(X,x_0)\to F_N$
	by composing with an automorphism of $F_N$, we will vary $[\phi]\in\outN$ in its conjugacy class. Thus the
	most one can ask for it terms of the uniqueness of $\phi$ is that the conjugacy class of $[\phi]\in\outN$
	should be unique.
	
	In order to establish that $\phi$ enjoys this degree of uniqueness,  
	we consider distinct free actions of $\Z/m$ on compact graphs of euler characteristic
	$1-N$, say $X$ and $X'$.
	We fix generating isometries $f$ and $f'$ for these actions, and we must construct an isomorphism $\pi_1X\to \pi_1X'$
	that conjugates the outer automorphism defined by $f$ to the outer automorphism defined by $f'$. (One could realise
	this by a homotopy equivalence, but there will not be a conjugating isometry in general.)
	
	To this end, we fix markings (isomorphisms)
	$\mu:F_M\to \pi_1(X/\<f\>)$ 
	and $\mu':F_M\to \pi_1(X'/\<f'\>)$ 
	and epimorphisms $p: \pi_1(X/\<f\>)\to \Z/m$
	and ${p': \pi_1(X'/\<f'\>)\to \Z/m}$, 
	the kernels of which give the covering spaces 
	$X\to X/\<f\>$ and $X'\to X'/\<f'\>$. 
	By shrinking a maximal tree in each, 
	we can assume that $X/\<f\>$ and $X'/\<f'\>$ 
	are $M$-roses, i.e. graphs with a single vertex. 
	(This amounts to moving between geometric representatives 
	of the given automorphisms in the fixed point set 
	in the spine of Culler and Vogtmann's Outer space 
	\cite{culler1986moduli}.) 
	We adjust our choice of markings so that both roses have edges labelled by the same basis $\{a_1,\dots,a_M\}$
	of $F_M$; here, implicity, we are moving between conjugates of the embeddings $\Z/m\hookrightarrow \outN$
	that are represented -- and note that it is really the embeddings, not the images 
	(we still have a preferred choice of generator 
	coming from $f$ and $f'$). 
	
	With these reductions, we have a fixed identification 
	that we regard as an equality $R_M=X/\<f\> = X'/\<f'\>$
	and $\pi_1(X/\<f\>)=F_M$ 
	but we are still carrying along distinct choices 
	of conjugacy classes $[t]$ and $[t']$ in $F_M$ 
	representing the deck transformations $f$ and $f'$. 
	Now, the covering spaces $X\to R_M$ and $X'\to R_M$ 
	correspond (under the standard Galois correspondence) 
	to epimorphisms $p:F_M\to \Z/m$ and $p':F_M\to\Z/m$ 
	such that $p(t)=\theta$ and $p'(t')=\theta$, 
	where $\theta$ is a fixed generator of $\Z/m$; 
	i.e. the fundamental groups of these coverings are 
	$K:=\ker p$ and $K':=\ker p'$,
	the action of $f$ on $X$ corresponds to 
	conjugation of $\ker p$ by $t$ 
	(equivalently, any element of $tK$), and 
	the action of $f'$ on $X'$ corresponds to 
	conjugation of $\ker p'$ by $t'$ 
	(equivalently, any element of $t'K'$). 
	
	The crucial point to observe now
	is that there is only one Nielsen equivalence class 
	of generating sets for $\Z/m$ of each cardinality $M\ge 2$ 
	(it is easy to prove this directly 
	using Euclid's algorithm). 
	Hence there is an automorphism $\psi:F_M\to F_M$ 
	such that $p'\circ\psi = p$. 
	Then, $K' = \psi(K)$ and $\psi(tK)=t'K'$; 
	say $\psi(t) = t'k'$ with $k'\in K'$. 
	
	The restriction of $\psi$ gives us an isomorphism 
	from $\pi_1X = K \cong F_N$ to $\pi_1X'=K'\cong F_N$ 
	that conjugates	the automorphism $\ad_{t}|_K$ of $K$ 
	to the automorphism $\ad_{t'k'}|_{K'}$ of $K'$; 
	the former, by definition, 
	lies in the outer automorphism class defined by $f$ 
	and the latter 
	is in the outer automorphism class defined by $f'$. 
\end{proof}

\begin{corollary}
	If an element  $\Phi\in\outN$ of finite order $m$ 
	has a geometric realisation $f:X\to X$ 
	such that the action of $\<f\>$ is free, 
	then $\Phi$ is conjugate in $\outN$ to $\Phi^j$ 
	for all integers $j$ coprime to $m$. 
\end{corollary}

This corollary becomes false if one assumes only that $f$ itself has no fixed points in $X$,
as our next construction shows.

\subsection{Automorphisms not conjugate to coprime powers}

The discussion that follows revolves around an understanding 
of the different geometric realisations $f:X\to X$
of a torsion element $\Phi\in\outN$. 
The first thing to note is that the discussion is 
about outer automorphisms rather than automorphisms, 
so we are not concerned with basepoints. 
The second point to note is that we are really only concerned 
with conjugacy classes in $\outN$,
which means that we do not fuss over details of the marking 
(i.e. the fixed isomorphism $\pi_1X\cong F_N$), 
since a change of marking replaces $\Phi$ by a conjugate. 
Nevertheless, a marking is needed 
if we want to talk about a specific $\Phi\in\outN$, 
and markings are also needed to pin down 
the connection between geometric realisations of $\Phi$ 
and the fixed points of $\Phi$ 
in the standard action of $\outN$ on the spine $K_N$ 
of Outer space \cite{culler1986moduli}. 
We will rely on the study of fixed-point sets by 
Krsti\'{c} \cite{krstic1989actions} 
and Kristi\'{c}--Vogtmann \cite{krstic1993equivariant}, 
who prove that ${\rm{Fix}}(\Phi)\subset K_N$ is connected 
and is the geometric realisation of the poset $X\preceq X'$ 
where $X$ and $X'$ are realisations of $\Phi$ 
such that $X$ is obtained from $X'$ 
by the collapse of a $\Phi$-invariant forest. 

We fix an odd integer $m\ge 5$ 
and an integer $r>1$ coprime to it. 
Let $X$ be the graph that has two circuits of length $m$ 
with the $i$-th vertex of the first (inner) circuit 
connected tothe $i$-th vertex of the second (outer) circuit 
by $r$ edges; 
the case $m = 5$ and $r = 3$ is shown in \Cref{mr-graph}. 
The rank of $\pi_1X$ is $mr + 1$ 
and the isometry group of $X$ is 
$$G=\big((
	\oplus_{i=1}^m\sym(r))\rtimes D_{2m}\big) \times (\Z/2),$$
where the action of $D_{2m}$ on $\oplus_{i=1}^m\sym(r))$
preserves both the inner and outer $m$-cycle, 
acting in the obvious way,
the $i$-th copy of $\sym(r)$ 
permutes the edges connecting 
the $i$-th vertices of the inner and outer circuits, 
and the central $1\times (\Z/2)$ 
interchanges the inner and outer cycles, 
sending each of the connecting edges to itself 
with reversed orientation. 

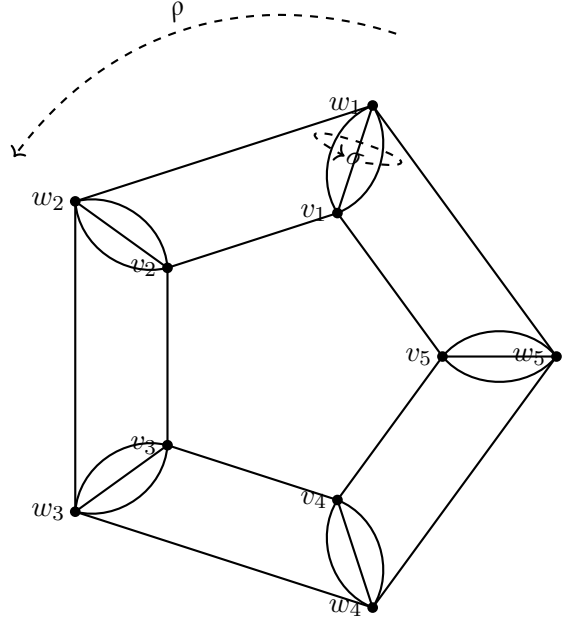
\begin{figure}
	\begin{tikzpicture}
	
	\def\m{5}
	\def\r{3}
	
	\def\skip{0}
	
	\def\radiusA{2}
	\def\radiusB{3.5}
	\def\radiusC{4.5} 
	\def\radiusD{2.75}
	
	\def\bigangle{50}
	
	\foreach \i in {1,...,\m} {
		\coordinate (A\i) at (360/\m * \i: \radiusA);
		\fill (A\i) circle (2pt);
		\node[left] at (A\i) {$v_{\i}$}; 
	}
	
	\foreach \i in {1,...,\m} {
		\coordinate (B\i) at (360/\m * \i: \radiusB);
		\fill (B\i) circle (2pt);
		\node[left] at (B\i) {$w_{\i}$}; 
	}
	
	\foreach \i in {1,...,\m} {
		\foreach \j in {1,...,\r} {
			\pgfmathsetmacro{\angle}{-\bigangle + 2*\bigangle*(\j-1)/(\r-1)}
			\draw[thick] (A\i) to [bend left=\angle] (B\i);
		}
	}
	
	\foreach \i in {1,...,\m} {
		\pgfmathtruncatemacro{\nexti}{mod(\i,\m) + 1}
		\pgfmathtruncatemacro{\skippednexti}{mod(\i+\skip,\m) + 1}
		\draw[thick] (A\i) -- (A\skippednexti);  
		\draw[thick] (B\i) -- (B\nexti);         
	}
	
	\foreach \i in {1,2} {
		\coordinate (C\i) at (360/\m * \i: \radiusC);
	}
	
	\draw[thick, ->, bend right=180/\m, dashed] (C1) to node[midway, above, sloped] {$\rho$} (C2);
	
	\draw[thick, ->, dashed, bend right=(180-0.3*180/\m), looseness=15] 
	(360/\m - 3: \radiusD) to node[midway, below, sloped] {$\sigma$} (360/\m + 3: \radiusD);
	
\end{tikzpicture}
\caption{The graph $X$ with labelled vertices.}
\label{mr-graph}
\end{figure}

\begin{figure}
\centering
\def\m{5}
\def\r{3}

\def\skip{0}

\def\radiusA{1}
\def\radiusB{1.75}
\def\radiusC{2.25} 
\def\radiusD{0.75}

\def\bigangle{50}

\subfigure[Graph $X$.]{
	\begin{tikzpicture}
		\foreach \i in {1,...,\m} {
			\coordinate (A\i) at (360/\m * \i: \radiusA);
			\fill (A\i) circle (2pt);
		}
		
		\foreach \i in {1,...,\m} {
			\coordinate (B\i) at (360/\m * \i: \radiusB);
			\fill (B\i) circle (2pt);
		}
		
		\foreach \i in {1,...,\m} {
			\foreach \j in {1,...,\r} {
				\pgfmathsetmacro{\angle}{-\bigangle + 2*\bigangle*(\j-1)/(\r-1)}
				\draw[thick] (A\i) to [bend left=\angle] (B\i);
			}
		}
		
		\foreach \i in {1,...,\m} {
			\pgfmathtruncatemacro{\nexti}{mod(\i,\m) + 1}
			\draw[thick] (A\i) -- (A\nexti);  
			\draw[thick] (B\i) -- (B\nexti);         
		}
	\end{tikzpicture}
}
\subfigure[Graph $X^+$.]{
	\begin{tikzpicture}
		\foreach \i in {1,...,\m} {
			\coordinate (A\i) at (360/\m * \i: \radiusA);
			\fill (A\i) circle (2pt);
		}
		
		\foreach \i in {1,...,\m} {
			\coordinate (B\i) at (360/\m * \i: \radiusB);
			\fill (B\i) circle (2pt);
		}
		
		\foreach \i in {1,...,\m} {
			\coordinate (C\i) at (360/\m * \i: \radiusC);
			\fill (C\i) circle (2pt);
		}
		
		\foreach \i in {1,...,\m} {
			\foreach \j in {1,...,\r} {
				\pgfmathsetmacro{\angle}{-\bigangle + 2*\bigangle*(\j-1)/(\r-1)}
				\draw[thick] (A\i) to [bend left=\angle] (B\i);
			}
		}
		
		\foreach \i in {1,...,\m} {
			\pgfmathtruncatemacro{\nexti}{mod(\i,\m) + 1}
			\draw[thick] (A\i) -- (A\nexti);  
			\draw[thick] (B\i) -- (C\i);
			\draw[thick] (C\i) -- (C\nexti);         
		}
	\end{tikzpicture}
}
\\
\subfigure[Graph $X^-.$]{
	\begin{tikzpicture}
		\foreach \i in {1,...,\m} {
			\coordinate (A\i) at (360/\m * \i: \radiusA);
			\fill (A\i) circle (2pt);
		}
		
		\foreach \i in {1,...,\m} {
			\coordinate (B\i) at (360/\m * \i: \radiusB);
			\fill (B\i) circle (2pt);
		}
		
		\foreach \i in {1,...,\m} {
			\coordinate (D\i) at (360/\m * \i: \radiusD);
			\fill (D\i) circle (2pt);
		}
		
		\foreach \i in {1,...,\m} {
			\foreach \j in {1,...,\r} {
				\pgfmathsetmacro{\angle}{-\bigangle + 2*\bigangle*(\j-1)/(\r-1)}
				\draw[thick] (A\i) to [bend left=\angle] (B\i);
			}
		}
		
		\foreach \i in {1,...,\m} {
			\pgfmathtruncatemacro{\nexti}{mod(\i,\m) + 1}
			\draw[thick] (B\i) -- (B\nexti);  
			\draw[thick] (A\i) -- (D\i);
			\draw[thick] (D\i) -- (D\nexti);         
		}
	\end{tikzpicture}
}
\subfigure[Graph $X^\pm$.]{
	\begin{tikzpicture}
		\foreach \i in {1,...,\m} {
			\coordinate (A\i) at (360/\m * \i: \radiusA);
			\fill (A\i) circle (2pt);
		}
		
		\foreach \i in {1,...,\m} {
			\coordinate (B\i) at (360/\m * \i: \radiusB);
			\fill (B\i) circle (2pt);
		}
		
		\foreach \i in {1,...,\m} {
			\coordinate (C\i) at (360/\m * \i: \radiusC);
			\fill (C\i) circle (2pt);
			\coordinate (D\i) at (360/\m * \i: \radiusD);
			\fill (D\i) circle (2pt);
		}
		
		\foreach \i in {1,...,\m} {
			\foreach \j in {1,...,\r} {
				\pgfmathsetmacro{\angle}{-\bigangle + 2*\bigangle*(\j-1)/(\r-1)}
				\draw[thick] (A\i) to [bend left=\angle] (B\i);
			}
		}
		
		\foreach \i in {1,...,\m} {
			\pgfmathtruncatemacro{\nexti}{mod(\i,\m) + 1}
			\draw[thick] (D\i) -- (D\nexti);  
			\draw[thick] (A\i) -- (D\i);
			\draw[thick] (C\i) -- (C\nexti);         
			\draw[thick] (B\i) -- (C\i);
		}
	\end{tikzpicture}
}
	\caption{Graphs fixed by the action of $\<\alpha\>$ on Outer space.}
	\label{fig:realisations}
\end{figure}
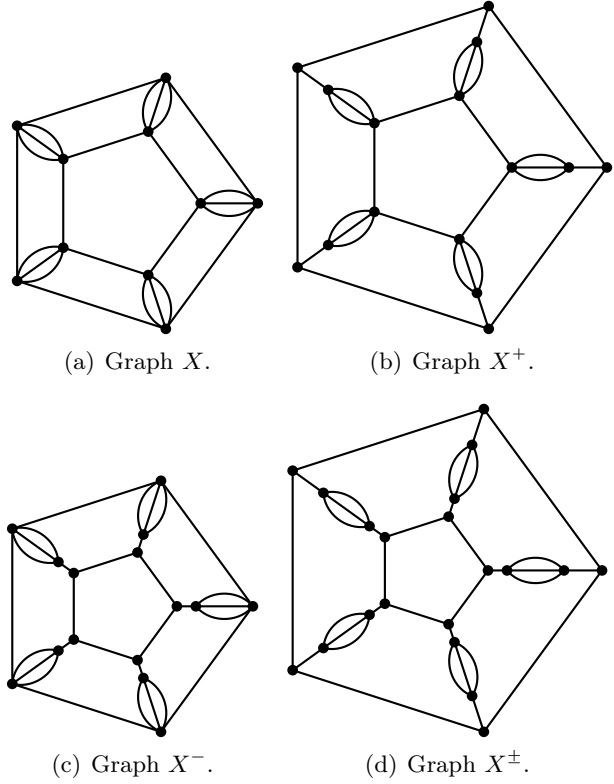

Consider the isometry $\alpha$ 
of $X$ that rotates the long cycles by one click $i\mapsto i+1 \mod m$ and acts as a cycle 
$\sigma\in \sym(r)$ of
length $r$ on each of the cages of $r$ edges. In the displayed decomposition of 
$$G=(\oplus_{i=1}^m\sym(r)\rtimes D_{2m})\times (\Z/2),$$
this has the form $((\sigma,\dots,\sigma), \rho, 0)$ 
where $\rho$ is the $m$-cycle $(1\, 2\, \dots \, m)$.

\smallskip

We would like to say that  $\alpha:X\to X$ is 
the unique geometric realisation 
of the conjugacy class in $\outN$ that it represents, 
where $N=mr+1$ 
-- this is almost true, but not quite. 
We also have to consider the graph $X^+$ obtained from $X$ 
by blowing-up the vertices in the outer circuit, 
separating it from the incident $r$-cage with a new edge; 
we have to consider the graph $X^-$ obtained from $X$ 
by modifying each vertex of the inner circuit instead; 
and we need the graph $X^\pm$
that is obtained from $X$ 
by performing both sets of blow-ups. 
All of the graphs $X, X^+, X^-, X^\pm$ are shown 
in \Cref{fig:realisations}. 
It is clear that each set of blow-ups 
can be performed in an $\alpha$-equivariant manner. 

\begin{lemma} 
\label{l1}
	Let $\alpha:X\to X$ be the isometry defined above, fix an identification $\pi_1X\cong F_N$
	and let $\Phi\in\outN$ be the outer automorphism defined by $\alpha$. Then, 
	\begin{enumerate}
		\item 
			$\alpha$ (and hence $\Phi$) has order $mr$.
		\item 
			In $G ={{\rm{Isom}}}(X)$, 
			the elements $\alpha$ and $\alpha^j$ are conjugate 
			if and only if $j \equiv \pm 1 \mod m$. 
		\item 
			The only geometric realisations of $\Phi$ 
			are $X,\, X^+,\, X^-$ and $X^\pm$. 
	\end{enumerate} 
\end{lemma}

\begin{proof}
	(1) is immediate, since $m$ and $r$ are coprime. 
	
	For (2), note that $G$ retracts onto $D_{2m}$, 
	and any conjugation in $D_{2m}$ 
	sends the $m$-cycle $\rho$ to $\rho$ or $\rho^{-1}$, 
	so $\alpha$ can only be conjugate to $\alpha^j$ 
	if $j \equiv \pm 1 \mod m$. 
	On the other hand, for $j \equiv \pm 1 \mod m$ we get 
	$\alpha^j = \big((\sigma^j, \dots, \sigma^j), \rho^{j}, 0)$ 
	which equals 
	\begin{equation*}
		\big((\tau, \dots, \tau), c, 0\big) \cdot 
		\big((\sigma,\dots,  \sigma), \rho, 0\big) \cdot
		\big((\tau^{-1}, \dots, \tau^{-1}), c^{-1}, 0\big)
	\end{equation*}
	where $\tau \sigma \tau^{-1} = \sigma^j$ in $\sym(r)$ 
	and $c \rho c^{-1} = \rho^j$ in $D_{2m}$;
	such $\tau \in \sym(m)$ exist 
	because $\sigma$ and $\sigma^j$ are both $r$-cycles, 
	and for $c$ we take $c=1$ if $j\equiv 1$ 
	and $c\not\in\<\rho\>$ if $j\cong -1$. 
	Thus $\alpha$ and $\alpha^j$ are conjugate in $G$, 
	since $\big((\tau^{-1},\dots, \tau^{-1}), c^{-1}, 1) = 
		\big((\tau,\dots,\tau), c, 1\big)^{-1}$. 
	
	The work of Krsti\'{c} \cite{krstic1989actions} 
	and Kristi\'{c}-Vogtmann \cite{krstic1993equivariant}, 
	discussed above, translates (3) 
	into a three-part assertion: 
	\begin{enumerate}[label=(\roman*)]
		\item 
			that $X$ does not contain 
			a non-trivial $\alpha$-invariant forest 
			(so no equivariant collapses are possible), 
		\item 
			that $X^+, X^-$ and $X^\pm$ are 
			the only $\Phi$-equivariant blow-ups of $X$, 
		\item 
			the only forest in each of these graphs 
			that is invariant under the action of 
			$\<\Phi\>\cong\Z/{mr}$
			is the one that was blown-up from $X$. 
			(Recall that $\Phi$ is acting as $\alpha$ on $X$, 
			and it acts as the obvious extension of $\alpha$ 
			on the other graphs.)
	\end{enumerate}

	Assertions (i) and (iii) are obvious from inspection. 
	For (ii), note that 
	the only $\alpha$-invariant sets of vertices in $X$ 
	are the vertex set of the outer cycle, 
	the vertex set of the inner cycle, 
	and the full vertex set, 
	so any equivariant blow-up has to be based on one of these
	sets. Moreover,  the stabiliser in $\<\alpha\>\cong \Z/mr$ of each vertex acts transitively on the edges of the $r$-cage
	incident at that vertex, so the endpoints of these edges cannot be separated in any equivariant blow-up. 
	Each vertex in a blow-up
	has to have valence at least three, so $X^+, X^-$ and $X^\pm$ are the only possibilities. 
\end{proof}

We dispense with the abbreviation $N=mr+1$. 

\begin{proposition}
\label{p:not-conj}
	$\Phi$ is conjugate to $\Phi^j$ in ${\rm{Out}}(F_{mr+1})$ 
	if and only if $j \equiv \pm 1 \mod m$. 
\end{proposition}

\begin{proof} 
	If $j$ is not coprime to $mr$, 
	then $\Phi^j$ is not conjugate to $\Phi$ 
	because	they have different orders. 
	If $j$ is coprime to $mr$, then $\<\Phi^j\>=\<\Phi\>$. 
	Thus any element of $\out(F)$ 
	that conjugates $\Phi$ to $\Phi^j$ 
	will normalise $\<\Phi\>$ 
	and hence leave the fixed point set of $\<\Phi\>$ 
	in the spine of Outer Space invariant. 
	We have argued that this fixed point set 
	has four vertices, $X,\, X^+,\, X^-,\, X^\pm$, 
	with edges (representing equivariant forest-collapses) 
	connecting $X$ and $X^\pm$ to each other 
	and to the other vertices $X^+$ and $X^-$. 
	The distinctive feature of $X$ is that 
	it is the only graph with no invariant forest 
	and the action of $\out(F_{mr+1})$ doesn't change this. 
	Thus the normaliser 
	of $\<\Phi\>$ in ${\rm{Out}}(F_{mr+1})$ 
	must fix the vertex $X$, 
	which means that $\Phi$ and $\Phi^j$ 
	will be conjugate in ${\rm{Out}}(F_{mr+1})$ if and only if 
	they are conjugate in $G={\rm{Stab}}(X)$, 
	and \Cref{l1} applies. 
\end{proof}

Combining \Cref{p:not-conj} with \Cref{c:coprime-iso} we have: 

\begin{proposition} 
	Let $m\ge 5$ be an odd integer, 
	let $r$ be an odd integer coprime to $m$, 
	and let $\Phi = [\phi]\in {\rm{Out}}(F_{N})$ be as above, 
	with $N=mr+1$. 
	Then $F_N\rtimes_{\phi}\Z\cong F_N\rtimes_{\phi^2}\Z$ 
	but $\Phi$ is not conjugate	to $\Phi^2$ 
	and hence there is no isomorphism 
	$F_N\rtimes_{\phi}\Z\to F_N\rtimes_{\phi^2}\Z$ 
	preserving $F_N$.
\end{proposition}

\begin{example} 
	We required $r$ to be odd in the preceding result 
	so that we could take $j=2$, 
	but to minimise the rank $N$ 
	we should take $m=5$ and $r=2$. 
	Then we obtain $[\phi]=\Phi\in\out(F_{11})$ of order $10$ 
	with $\Phi^3$ not conjugate to $\Phi$. 
	In this case, 
	$F_{11}\rtimes_\phi\Z$ and $F_{11}\rtimes_{\phi^3}\Z$
	are both isomorphic to 
	$$
	\< x, y, s, t \mid  [x,s]=1=[y,t],\ x^2=y^2 \>.
	$$
\end{example}

\subsection{
	Isomorphisms derived from non-conjugate cyclic subgroups}

We now consider a variant $X'$ of the graph $X$ 
in which the inner cycle is replaced 
by a cycle that connects cages a distance 2 apart; 
the case $m=5$ and $r=3$ is shown in \Cref{mr-graph-twisted}. 
This is not one of the graphs 
considered in the previous section,
so it does not support a geometric realisation 
of any conjugate of $\Phi\in {\rm{Out}}(F_{mr+1})$; 
in other words, 
it is not the underlying graph 
of a vertex in the spine of Outer Space fixed by $\Phi$. 

Like $X$, the graph $X'$ has an isometry of order $mr$ 
that rotates the picture by one click 
while permuting the edges of each $r$-cage 
with a cycle of length $r$; 
call this isometry $\beta$ 
and let $\Psi\in {\rm{Out}}(F_{mr+1})$ be 
the outer automorphism defined by $\beta$ 
(with some marking $\pi_1X'\cong F_{mr+1}$). 

The key feature of this construction is 
that the mapping cylinders of 
${\alpha : X\to X}$ and $\beta: X'\to X'$ are homeomorphic. 
The easiest way to see this 
is to recall from the proof of \Cref{prop:graph_of_Zs} 
the description of the mapping torus $M_\alpha$ 
as a graph of spaces. 
This decomposition depends only 
on the adjacencies of $\alpha$-orbits in $X$; 
and the pattern of adjacencies for the $\beta$-orbits in $X'$ 
is identical. 
In each case there are two orbits of vertices 
and three orbits of edges; 
the underlying integer-labelled graph 
has two loops with ends labelled $1$ 
connected by an edge that has both ends labelled $r$. 
Thus both mapping cylinders 
are copies of the standard 2-complex for the presentation 
$$
	F_{mr+1}\rtimes_\phi\Z\cong F_{mr+1}\rtimes_\psi\Z \cong 
	\< x,y,s,t \mid [x,s]=1=[y,t],\, x^{r}=y^{r}\>.
$$
But $\<\Phi\>\cong\Z/mr$ 
is not conjugate to $\<\Psi\>\cong\Z/mr$ 
in ${\rm{Out}}(F_{mr+1})$. \qed

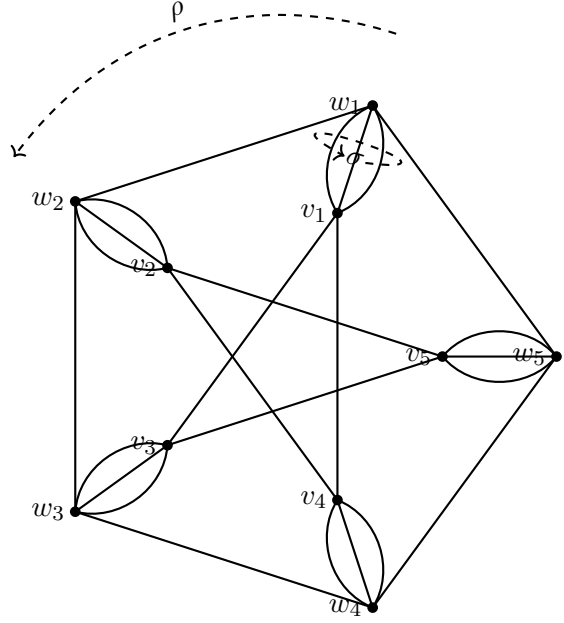
\begin{figure}
	\begin{center}
		\begin{tikzpicture}
	
	\def\m{5}
	\def\r{3}
	
	\def\skip{1}
	
	\def\radiusA{2}
	\def\radiusB{3.5}
	\def\radiusC{4.5} 
	\def\radiusD{2.75}
	
	\def\bigangle{50}
	
	\foreach \i in {1,...,\m} {
		\coordinate (A\i) at (360/\m * \i: \radiusA);
		\fill (A\i) circle (2pt);
		\node[left] at (A\i) {$v_{\i}$}; 
	}
	
	\foreach \i in {1,...,\m} {
		\coordinate (B\i) at (360/\m * \i: \radiusB);
		\fill (B\i) circle (2pt);
		\node[left] at (B\i) {$w_{\i}$}; 
	}
	
	\foreach \i in {1,...,\m} {
		\foreach \j in {1,...,\r} {
			\pgfmathsetmacro{\angle}{-\bigangle + 2*\bigangle*(\j-1)/(\r-1)}
			\draw[thick] (A\i) to [bend left=\angle] (B\i);
		}
	}
	
	\foreach \i in {1,...,\m} {
		\pgfmathtruncatemacro{\nexti}{mod(\i,\m) + 1}
		\pgfmathtruncatemacro{\skippednexti}{mod(\i+\skip,\m) + 1}
		\draw[thick] (A\i) -- (A\skippednexti);  
		\draw[thick] (B\i) -- (B\nexti);         
	}
	
	\foreach \i in {1,2} {
		\coordinate (C\i) at (360/\m * \i: \radiusC);
	}
	
	\draw[thick, ->, bend right=180/\m, dashed] (C1) to node[midway, above, sloped] {$\rho$} (C2);
	
	\draw[thick, ->, dashed, bend right=(180-0.3*180/\m), looseness=15] 
	(360/\m - 3: \radiusD) to node[midway, below, sloped] {$\sigma$} (360/\m + 3: \radiusD);
	
\end{tikzpicture}
	\end{center}
	\caption{The twisted graph $X$.}
	\label{mr-graph-twisted}
\end{figure}

\bigskip

We close by noting that 
the order of $[\phi]$ is an invariant of $F_N\rtimes_\phi\Z$, 
provided that $N$ is fixed. 
This is a strong condition, 
as it was shown by Button \cite{button2007mapping} 
that a free-by-cyclic group with first Betti number $\ge 2$ 
is isomorphic to $F_M \rtimes \Z$ 
for infinitely many $M$. 

\begin{proposition}
\label{p:order-only}
	If $F_N\rtimes_\phi\Z\cong F_N\rtimes_\psi\Z$ and $[\phi]\in\outN$ has order $m$, then $[\psi]\in\outN$ has order $m$
\end{proposition}

\begin{proof} 
	Let $\G=F_N\rtimes_\phi\Z$.
	It follows from \Cref{free-by-cyclic-with-centre} 
	that $\G/Z(\G)$ contains $F_N$ as a subgroup of index $m$, 
	so the rational Euler characteristic of $\G/Z(\G)$ 
	is $(1-N)/m$. 
\end{proof}


\bibliographystyle{alpha}
\bibliography{Bibliography}

\begin{thebibliography}{BMRS21}

\bibitem[AP24]{AP}
Naomi Andrews and Irakli Patchkoria.
\newblock {On the Farrell--Tate K-Theory of ${{\rm{Out}}}(F_n)$}.
\newblock {\em In preparation}, 2024.

\bibitem[Bau69]{baumslag1969non}
Gilbert Baumslag.
\newblock A non-cyclic one-relator group all of whose finite quotients are
  cyclic.
\newblock {\em J Aust. Math. Soc.}, 10(3-4):497--498, 1969.

\bibitem[Bau74]{baumslag1974residually}
Gilbert Baumslag.
\newblock Residually finite groups with the same finite images.
\newblock {\em Compositio Mathematica}, 29(3):249--252, 1974.

\bibitem[BCR16]{bridson2013determining}
M.~R. Bridson, M.~D.~E. Conder, and A.~W. Reid.
\newblock Determining {F}uchsian groups by their finite quotients.
\newblock {\em Israel J. Math.}, 214(1):1--41, 2016.

\bibitem[Bes02]{bestvina2003topology}
Mladen Bestvina.
\newblock The topology of {${\rm Out}(F_n)$}.
\newblock In {\em Proceedings of the {I}nternational {C}ongress of
  {M}athematicians, {V}ol. {II} ({B}eijing, 2002)}, pages 373--384. Higher Ed.
  Press, Beijing, 2002.

\bibitem[BMRS21]{bridson2021profinite}
Martin~R. Bridson, D.B. McReynolds, Alan~W. Reid, and Ryan Spitler.
\newblock On the profinite rigidity of triangle groups.
\newblock {\em Bull. London Math. Soc.}, 53(6):1849--1862, 2021.

\bibitem[BMT07]{BaumMillerTroeger}
Gilbert Baumslag, Charles~F. Miller, III, and Douglas Troeger.
\newblock Reflections on the residual finiteness of one-relator groups.
\newblock {\em Groups Geom. Dyn.}, 1(3):209--219, 2007.

\bibitem[BRW17]{bridson2017profinite}
Martin~R. Bridson, Alan~W. Reid, and Henry Wilton.
\newblock Profinite rigidity and surface bundles over the circle.
\newblock {\em Bulletin of the London Mathematical Society}, 49(5):831--841,
  2017.

\bibitem[BT67]{baumslag1967centre}
Gilbert Baumslag and Tekla Taylor.
\newblock The centre of groups with one defining relator.
\newblock {\em Math. Ann.}, 175(4):315--319, 1967.

\bibitem[But07]{button2007mapping}
Jack~O. Button.
\newblock Mapping tori with first betti number at least two.
\newblock {\em J Math. Soc. Japan}, 59(2):351--370, 2007.

\bibitem[BV06]{bridson2006automorphism}
Martin~R. Bridson and Karen Vogtmann.
\newblock Automorphism groups of free groups, surface groups and free abelian
  groups.
\newblock In {\em Problems on mapping class groups and related topics},
  volume~74 of {\em Proc. Sympos. Pure Math.}, pages 301--316. Amer. Math.
  Soc., Providence, RI, 2006.

\bibitem[CHMV]{corson2023higman}
Samuel~M. Corson, Sam Hughes, Philip M{\"o}ller, and Olga Varghese.
\newblock Higman-{T}hompson groups and profinite properties of right-angled
  {C}oxeter groups.
\newblock {\em arXiv:2309.06213}.

\bibitem[CL16]{cashenlevitt}
Christopher~H. Cashen and Gilbert Levitt.
\newblock Mapping tori of free group automorphisms, and the
  {B}ieri-{N}eumann-{S}trebel invariant of graphs of groups.
\newblock {\em J. Group Theory}, 19(2):191--216, 2016.

\bibitem[Cul84]{culler1984finite}
Marc Culler.
\newblock Finite groups of outer automorphisms of a free group.
\newblock In {\em Contributions to group theory}, volume~33 of {\em Contemp.
  Math.}, pages 197--207. Amer. Math. Soc., Providence, RI, 1984.

\bibitem[CV86]{culler1986moduli}
Marc Culler and Karen Vogtmann.
\newblock Moduli of graphs and automorphisms of free groups.
\newblock {\em Invent. Math.}, 84(1):91--119, 1986.

\bibitem[CZ15]{chagas2015subgroup}
S.~C. Chagas and P.~A. Zalesskii.
\newblock Subgroup conjugacy separability of free-by-finite groups.
\newblock {\em Arch. Math. (Basel)}, 104(2):101--109, 2015.

\bibitem[For02]{forester2002deformation}
Max Forester.
\newblock Deformation and rigidity of simplicial group actions on trees.
\newblock {\em Geom. Topol.}, 6(1):219--267, 2002.

\bibitem[Fun13]{funar2013torus}
Louis Funar.
\newblock {Torus bundles not distinguished by TQFT invariants}.
\newblock {\em Geom. Topol.}, 17(4):2289--2344, 2013.

\bibitem[Gab02]{gaboriau}
Damien Gaboriau.
\newblock Invariants {$l^2$} de relations d'\'equivalence et de groupes.
\newblock {\em Publ. Math. Inst. Hautes \'Etudes Sci.}, 95:93--150, 2002.

\bibitem[Gro74]{grossman1974residual}
Edna~K Grossman.
\newblock On the residual finiteness of certain mapping class groups.
\newblock {\em Journal of the London Mathematical Society}, 2(1):160--164,
  1974.

\bibitem[GZ11]{grunewald2011genus}
Fritz Grunewald and Pavel Zalesskii.
\newblock Genus for groups.
\newblock {\em J. Alg.}, 326(1):130--168, 2011.

\bibitem[Hem14]{hempel2014some}
John Hempel.
\newblock Some 3-manifold groups with the same finite quotients.
\newblock {\em arXiv preprint arXiv:1409.3509}, 2014.

\bibitem[Hig51]{higman1951finitely}
Graham Higman.
\newblock A finitely generated infinite simple group.
\newblock {\em J London Math. Soc.}, 1(1):61--64, 1951.

\bibitem[HK23]{hughes2023profinite}
Sam Hughes and Monika Kudlinska.
\newblock On profinite rigidity amongst free-by-cyclic groups {I}: the generic
  case.
\newblock {\em arXiv preprint arXiv:2303.16834}, 2023.

\bibitem[Khr90]{khramtsov}
D.~G. Khramtsov.
\newblock Outer automorphisms of free groups.
\newblock In {\em Group-theoretic investigations ({R}ussian)}, pages 95--127.
  Akad. Nauk SSSR Ural. Otdel., Sverdlovsk, 1990.

\bibitem[Krs89]{krstic1989actions}
Sava Krsti{\'c}.
\newblock Actions of finite groups on graphs and related automorphisms of free
  groups.
\newblock {\em J Alg.}, 124(1):119--138, 1989.

\bibitem[KV93]{krstic1993equivariant}
Sava Krsti{\'c} and Karen Vogtmann.
\newblock Equivariant outer space and automorphisms of free-by-finite groups.
\newblock {\em Comm. Math. Helv.}, 68(1):216--262, 1993.

\bibitem[Lev15]{levitt2015generalized}
Gilbert Levitt.
\newblock Generalized baumslag--solitar groups: rank and finite index
  subgroups.
\newblock {\em Ann. l'Instit. Fourier}, 65(2):725--762, 2015.

\bibitem[Liu23]{liu2023finite}
Yi~Liu.
\newblock Finite-volume hyperbolic 3-manifolds are almost determined by their
  finite quotient groups.
\newblock {\em Invent. {M}ath.}, 231(2):741--804, 2023.

\bibitem[L{\"u}c94]{luck1994approximating}
Wolfgang L{\"u}ck.
\newblock Approximating l 2-invariants by their finite-dimensional analogues.
\newblock {\em Geom. Functi. Anal. (GAFA)}, 4:455--481, 1994.

\bibitem[Mag32]{magnus}
Wilhelm Magnus.
\newblock Das identita\"{a}tsproblem fu\"{u}r gruppen mit einer definierenden
  relation.
\newblock {\em Math. Ann.}, 106:295–307, 1932.

\bibitem[NS03]{nikolov2003finite}
Nikolay Nikolov and Dan Segal.
\newblock Finite index subgroups in profinite groups.
\newblock {\em Comptes Rendus Mathematique}, 337(5):303--308, 2003.

\bibitem[Rib17]{ribes2017profinite}
Luis Ribes.
\newblock {\em Profinite graphs and groups}.
\newblock Springer, 2017.

\bibitem[RZ14]{ribes2014normalizers}
Luis Ribes and Pavel~A. Zalesskii.
\newblock Normalizers in groups and in their profinite completions.
\newblock {\em Rev. Mat. Iberoamericana}, 30(1):165--190, 2014.

\bibitem[Ser64]{serre1964exemples}
Jean-Pierre Serre.
\newblock Exemples de vari{\'e}t{\'e}s projectives conjugu{\'e}es non
  hom{\'e}omorphes.
\newblock {\em C.R. Acad. Sci.}, 258(17):4194--4196, 1964.

\bibitem[Ser80]{serre1980trees}
Jean-Pierre Serre.
\newblock {\em Trees}.
\newblock Springer-Verlag, Berlin-New York, 1980.
\newblock Translated from the French by John Stillwell.

\bibitem[Ste72]{stebe1972conjugacy}
P.F. Stebe.
\newblock Conjugacy separability of certain {F}uchsian groups.
\newblock {\em Trans. American Math. Soc.}, 163:173--188, 1972.

\bibitem[SW79]{scott1979topological}
Peter Scott and Terry Wall.
\newblock Topological methods in group theory.
\newblock In {\em Homological group theory (Proc. Sympos., Durham, 1977)},
  volume~36, pages 137--203, 1979.

\end{thebibliography}

\end{document}